\documentclass[10pt,compress]{article}

\usepackage{amsfonts,amssymb,amsmath,amsthm,indentfirst}
\usepackage{fontenc,enumerate}
\usepackage{graphics}
\usepackage{color,graphicx}
\usepackage[dvipsnames]{xcolor}
\usepackage[colorlinks, linkcolor=red, citecolor=Green, urlcolor=OliveGreen, pagebackref, hypertexnames=false]{hyperref}

\usepackage{latexsym,mathrsfs,enumerate,verbatim,bm}

\usepackage[light]{antpolt}    \usepackage[T1]{fontenc}

\usepackage{wasysym}
\usepackage{graphics}

\usepackage{color,graphicx}

\newcommand{\blue}[1]{\textcolor{blue}{#1}}

\DeclareMathAlphabet{\mathpzc}{OT1}{pzc}{m}{it}

\renewcommand{\d}{\;{\rm d}}

\newcommand{\scal}[1]{\left\langle #1 \right\rangle}
\newcommand{\sett}[1]{\left\{   #1   \right\}}
\newcommand{\norm}[1]{\left\|   #1   \right\|}
\newcommand{\abso}[1]{\left|   #1   \right|}
\newcommand{\paar}[1]{\left(   #1   \right)}

\newcommand{\Bigabs}[1]{\Bigl\vert #1 \Bigr\vert}

\usepackage{xcolor}
\usepackage[pagewise,mathlines]{lineno}

\newcommand{\vertiii}[1]{{\left\vert\kern-0.25ex\left\vert\kern-0.25ex\left\vert #1 
		\right\vert\kern-0.25ex\right\vert\kern-0.25ex\right\vert}}
\renewcommand{\theequation}{\arabic{section}.\arabic{equation}}
\numberwithin{equation}{section}
\newtheorem{theorem}{\quad Theorem}[section]
\newtheorem{lemma}[theorem]{\quad Lemma}
\newtheorem{corollary}[theorem]{\quad Corollary}
\newtheorem{remark}[theorem]{\quad Remark}

\newcommand{\ngg}{\mathscr{G}}

\newcommand{\R}{\mathbb{R}}
\DeclareMathOperator{\supp}{supp}

\newcommand{\Z}{\mathbb{Z}}

\newcommand{\angles}[1]{\langle #1 \rangle}

\newcommand{\dd}{{\rm d}}
\newcommand{\dx}{\;{\rm d}x}
\newcommand{\ff}{\varphi}
\newcommand{\ee}{{\rm e}}
\newcommand{\ii}{{\rm i}}
\newcommand{\la}{\left\langle}
\newcommand{\ra}{\right\rangle}

\newcommand{\N}{\mathbb{N}}
\newcommand{\da}{\Delta}

\newcommand{\e}{\varepsilon}

\newcommand{\lam}{\lambda}
\newcommand{\x}{{X_\beta}}
\newcommand{\xx}{{\dot{X}_\beta}}

\newcommand{\lt}{{L^2(\mathbb{R}^n)}}
\newcommand{\vr}{{\varrho}}

\newcommand{\al}{\alpha}
\newcommand{\rr}{\mathbb{R}}
\newcommand{\rn}{{\mathbb{R}^n}}
\newcommand{\E}{\mathbb{E}}
\newcommand{\F}{\mathbb{F}}

\allowdisplaybreaks

\author{  Amin Esfahani\footnote{Department of Mathematics, Nazarbayev University, Astana 010000, Kazakhstan\newline   E-mail: amin.esfahani@nu.edu.kz, achenef.tesfahun@nu.edu.kz}~ and Achenef Tesfahun    }

\title{  The Cauchy problem for the   nonlinear Schrödinger equation   with a convolution potential 
	}

\date{}

\begin{document}
	\maketitle
	\begin{abstract}
		This paper investigates the nonlinear Schr\"{o}dinger equation with a singular convolution potential. It demonstrates the local well-posedness of this equation in a modified Sobolev space linked to the energy. Additionally, we derive conditions under which the solutions are uniformly bounded in the energy space. This finding is closely linked to the existence of standing waves for this equation.
	\end{abstract}

	
	
	
	
	\section{Introduction}
	
	In this paper, we consider the following Scr\"{o}dinger  equation
	\begin{equation}\label{rnls}
		\begin{cases}
			\ii u_t+\Delta u-\e L_V u +\varsigma  |u|^{2}u=0,\qquad \varsigma=\pm1 \\
			u(x,0)=u_0(x),\qquad x\in\rn,
		\end{cases}
	\end{equation}
	where, $\e\in\rr$,   $L_V u= V\ast u$ and $V$ is a potential. The choice of a convolution potential instead of the classical multiplicative potential is due to the fact that the term $L_V$ is diagonal in the Fourier basis (see \cite{bamgre,Faou-grebet, giul}).

	This equation appears in condensed matter physics and stochastic mechanics \cite{mendes}.
	Bourgain in \cite{bou-1}  proved
	existence of quasi-periodic solutions of \eqref{rnls} on the torus $\mathbb{T}^2$. He also studied a general class of the Schr\"{o}dinger equation in any dimension in \cite{bou-2}.
	The local well-posedness of a general class of \eqref{rnls} on the circle was studied in \cite{fei}.
	It was shown in \cite{guardia} that speed of growth of $H^s(\mathbb{T}^2)$-norms (with $s>1$) of the solutions of \eqref{rnls} with   the potential $V\in H^{s_0}(\mathbb{T}^2)$, with $s_0>0$, is the same as the one obtained for \eqref{rnls} without the potential. See also \cite{cgp}. 
	
	When $\e=0$, \eqref{rnls} turns into the classical nonlinear Schr\"{o}dinger (NLS) equation
	
	\begin{equation}\label{nls}
		\begin{cases}
			\ii u_t+\Delta u  +\varsigma  |u|^2u=0, \\
			u(x,0)=u_0(x)
		\end{cases}
	\end{equation}
	The plus and minus signs in front
	of the nonlinearity in both \eqref{rnls} and \eqref{nls} correspond to the focusing and defocusing cases, respectively. Contrary to \eqref{rnls}, equation \eqref{nls} enjoys the scaling property, that is: if $u$ is a solution of \eqref{nls} and $\lambda > 0$, then 
	$u_\lam(x,t)=\lam u(\lam x,\lambda^2 t)$ is also a solution of \eqref{nls}. So, the critical Sobolev exponent \eqref{nls} is  
	$s_c=\frac n2-1$ such that $\|u_\lam(0)\|_{H^{s_c}}=\|u(0)\|_{H^{s_c}}$
	for all $\lam>0$. Grunrock investigated the local well-posedness of \eqref{nls} in \cite{gro-1} using alternative function spaces below $L^2(\mathbb{R})$, while ensuring the local Lipschitz dependence on the initial data. In the focusing case, \eqref{nls} admits soliton and multisoliton solutions and is globally well-posed in $L^2$ due to the conservation of the $L^2$-norm.  Christ, Colliander, and Tao \cite{cct} demonstrated that the data-to-solution map is unbounded from $H^s(\mathbb{R})$ to $H^s(\mathbb{R})$ for $s<-1/2$. Recently, Harrop-Griffiths et al. in \cite{hrm} proved that \eqref{nls} is globally well-posed for all initial data in $H^s(\mathbb{R})$ with $s>-1/2$, wherein the solution map extends uniquely from Schwartz space to a jointly continuous map $\mathbb{R} \times H^s(\mathbb{R})\to H^s(\mathbb{R})$.  When $n=2,3, 4$, Strichartz estimates allow to prove local wellposedness in $H^s(\rn)$ for $s \geq s_c$, where the time of existence depends on the profile of the data as well as the norm (see \cite{cazenave-weissler}).
	
	Here, we consider the singular potential $V(x)=|x|^{2\beta-n}$ with $0<2\beta<n$, leading to $L_V$ transforming into $L_\beta=D^{-2\beta}$. It is noteworthy that \eqref{rnls} can be viewed as a perturbed version of \eqref{nls}, especially when the NLS equation is influenced by  a repulsive inverse-power potential, as depicted in:
	\begin{equation}\label{pnls}
		\begin{cases}
			\ii u_t+\Delta u  +\varsigma  |u|^2u=\varepsilon V(x)  u,\qquad  V(x)=|x|^{2\beta-n},\,\,2\beta>n-\min\{2,n\},\\
			u(x,0)=u_0(x).
		\end{cases}
	\end{equation}
	We refer the readers, for example, to works \cite{dinh, kmvzz} and references therein for recent progress on the well-posedness of \eqref{pnls} with different values of $\beta$.
	
	Our objective in this paper is to investigate the impact of the convolution term on the well-posedness of   \eqref{rnls}. To analyze   \eqref{rnls}, we aim to derive dispersive estimates that will allow us to determine a lower Sobolev index for well-posedness. Notably,   \eqref{rnls} conserves formally the functionals $\mathbb{P}$, $\F$, and $\E$, defined as follows:
	\[
	\mathbb{P}(u(t))=\Im  \scal{u(t),\nabla u(t)}_{L^2(\rn)}  ,
	\]
	\[
	\F(u(t))=\|u(t)\|_{L^2(\rn)}^2 ,
	\]
	\[
	\E(u(t))=\frac12\|\nabla u(t)\|_{L^2(\rn)}^2 
	+\frac{\e}{2}\la u,L_\beta u\ra_{L^2(\rn)}
	-
	\frac{\varsigma}{4}\|u(t)\|_{L^4(\rn)}^4  .
	\]
	Considering that the Fourier frequencies of $u$ are convolved with the potential, it is natural to explore the well-posedness in a modified Sobolev space where this convolution is controlled. In terms of the energy functional $\E$, the suitable energy space associated with   \eqref{rnls} is denoted by $\x=H^1\cap \dot{H}^{-\beta}$, and defined via 
	\[
	\|u\|_\x^2=\|u\|_{H^1(\rn)}^2+ \|L_{\beta/2}u\|_\lt^2
	=\norm{u}_{H^1}^2+\norm{D^{-\beta}u}_{L^2}^2.
	\]
	Since $X_\beta$ is a subspace of $H^1$ (see Lemma \ref{density-space}), then it is embedded into $L^{p+2}$ with 
	\[
	p\leq\begin{cases}
		\infty^-,&n\leq2,\\
		\frac{4}{n-2},&n>2.
	\end{cases}
	\]

	In the literature, it is well-established that the analysis of dispersive equations is closely tied to the phase function. One of the primary challenges in studying the Cauchy problem for \eqref{rnls} stems from its dispersion relation, given by $$m(\xi)=|\xi|^2+\e|\xi|^{-2\beta},$$ 
	which exhibits a singularity at the origin.
	In one-dimensional case $n=1$, if $\e>0$, then $| m''(\xi)|\gtrsim1$.  So it follows from the Van der Corput lemma (see Lemma \ref{lm-corput}) for $\e>0$ that
	\begin{equation}\label{osci-int}
		\sup_{x\in\rr}\left|\int_\rr\ee^{\ii(x\xi+ t m(\xi))}\dd\xi\right|\lesssim|t|^{-\frac12}.
	\end{equation}
	Hence, the Strichartz estimates associated with \eqref{rnls} are akin to those of \eqref{nls}, thereby establishing its local well-posedness in $L^2(\rr)$ (see \cite{linaresponce}). In the case $\e<0$, we observe that $m''(\xi)$ vanishes at some $\xi_0$, and $|m''(\xi)|\sim|\xi-\xi_0|$ in a small neighborhood of $\xi_0$. Consequently, the conditions of Theorem 2.1 in \cite{kpv} hold, and by an argument similar to that in the case $\e>0$, combined with the Strichartz estimates of \cite[Theorem 2.1]{kpv}, we establish local well-posedness in $L^2(\rr)$.
	
	In dimensions $n\geq2$, the local and global smoothing estimates of Kenig-Ponce-Vega \cite{kpv} prove challenging due to the inability to approximate the symbol $m$ by a polynomial. To address this issue, we exploit the radial symmetry of the symbol $m$, which allows us to reduce the problem to a one-dimensional oscillatory integral using the Bessel functions. This approach has been explored in prior research, including works by Cho and Lee \cite{cholee} and Guo, Peng, and Wang \cite{gpw}, where they considered general radial (nonhomogeneous) phase functions. However, these studies impose a crucial condition that $\frac{\dd^2m}{\dd r^2}\cdot\frac{\dd m}{\dd r}$ with $r=|\xi|$ should not change sign.
	On the contrary, we observe that the aforementioned condition for the phase function $m$ of \eqref{rnls} is not indeed valid, rendering the techniques of Cho and Lee \cite{cholee} seemingly inapplicable. To overcome this obstacle, we employ frequency localization by separating high and low frequencies, enabling us to establish the decay rate $t^{-\frac n2}$ for the localized $n$-dimensional version of \eqref{osci-int} (see Lemma \ref{lm-LocStr}). Subsequently, we derive suitable decay estimates using dyadic decomposition and the decay properties of the Bessel function. As a result, we note that the dispersive estimates of \eqref{rnls} are analogous to those of \eqref{nls}.

	To formulate our results, we define the modified Sobolev space $H_\beta^s$ via
	\begin{align*}
		\| u \|_{H^{s}_\beta} := \left\|  | \xi |^{-\beta}  \angles{ \xi }^{s+\beta}  \widehat{u} (\xi) \right\|_{L^2_{ \xi}}.
	\end{align*}
	Note that 	$$ \| u \|_{H^{1}_\beta} \sim \| u \|_{X_\beta}. $$
	
	Our first result is as follows.

	\begin{theorem}\label{thmlwp}
		Let  $n\geq 1$, $0< \beta <n/2$ and $\varepsilon\in\rr$. If  $s\geq(n-1)/2$  and $u_0\in H^s_\beta (\R^n)$, then
		the Cauchy problem \eqref{rnls} is locally well-posed.
	\end{theorem}

     	Now that we have established the local well-posedness of \eqref{rnls}, it is natural to inquire whether the solutions can be extended globally in time. The global well-posedness becomes an easy consequence of the energy conservation law provided in the defocusing case $\varsigma=-1$  and $\varepsilon>0$. However, in the focusing case $\varsigma=1$, we control the term $\|u\|_{L^4(\rn)}^4$ by a Sobolev-type embedding associated with \eqref{rnls} (see Lemma \ref{embeding-lemma}). This embedding produces a constant, the sharp constant of Lemma \ref{embeding-lemma}, which plays a crucial role in finding the conditions (see Theorems \ref{global conditions} and \ref{Interpolated global conditions}) under which the local solution is uniformly bounded in the energy space $\x$. Thus, in the quest for optimal conditions, we strive to determine the sharp constant of Lemma \ref{embeding-lemma}, which, in turn, is tightly connected to the existence of standing waves of the form $\exp(\ii\omega t)\ff(x)$ of \eqref{rnls}.
	\begin{theorem}\label{global conditions}
		Let $\psi$ and $W$ be the ground states of 
  	\begin{equation}\label{zero-beta-ground}
			-\Delta\psi+\psi=\psi^{p+1}
		\end{equation}
  and  	\begin{equation}\label{energy-critical}
			-\Delta W=W^{\frac{n+2}{n-2}}, 
		\end{equation}
		respectively.	Suppose \eqref{rnls} is locally well-posed in $X_\beta$. Moreover, assume 
		\begin{enumerate}[(I)]
			\item in the defocusing case $\varsigma=-1$ that $n\geq1$; or
			\item in the focusing case $\varsigma=+1$ that one of the following cases occurs:
			\begin{enumerate}[(i)]
				\item $n=1$;
				\item $n=2$ and $\F(u_0)<\F(\psi)$;
				\item $n=3$   and 
				\begin{equation}\label{cond2}
					\E(u_0)\F(u_0)
					<
					\E_0(\psi)\F(\psi)
				\end{equation}	
				\begin{equation}\label{cond1}
					\|u_0\|_{\dot{X}_\beta}^2\F (u_0)
					<
					\|\nabla\psi\|_\lt^2\F (\psi),
				\end{equation}
				where
				\[
				\E_0(u(t))=\frac12\int_\rn |\nabla u(x,t)\|^2 \dx
				-
				\frac{\varsigma}{4}\int_\rn | u(x,t)|^{4} \dx ;
				\]
				\item $n=4$ and
				\[
				\E(u_0)<\E_0(W),\qquad\|u_0\|_{\dot{X}_\beta}<\|\nabla W\|_\lt,
				\]
				\item 	  $1\leq n\leq 4$ and
				\begin{equation}\label{cond-mass-zero} 
					\begin{split}
						\|u_0\|_{\dot{X}_\beta} < C_{\beta,n,2} \|Q_\ast\|_{\dot{X}_\beta},\qquad
						\E(u_0)< C_{\beta,n,2} \E(Q_\ast),
					\end{split}
				\end{equation}
				where $Q_\ast$ is a ground state of \eqref{gs-mass-zero} with $p=2$ and $C_{\beta,n,2}$ is defined as in Lemma \ref{lemma-best-masszero}.
			\end{enumerate}	
		\end{enumerate}
		Then the local solution of \eqref{rnls} with initial data $u_0$ is uniformly bounded in $X_\beta$.
	\end{theorem}
	Furthermore, it is of independent interest to investigate the behavior of $\ff(x)$. The subsequent result is established by utilizing the properties of the Gelfand–Shilov space and the H\"{o}rmander class:
	
	\begin{theorem}\label{decay-theor}
		Let $\beta < n/2$. There exists $\beta_\ast$ such that for any $\omega > \beta_\ast$,  \eqref{rnls} possesses a nontrivial standing wave solution $\exp(\ii\omega t)\ff(x)$ with $\ff \in X_\beta$. Moreover, it holds that $\langle x \rangle^{n + 4\beta + 2}\ff \in L^\infty(\mathbb{R}^n)$. Additionally, $\ff$ and $L_\beta\ff$ belong to $H^\infty(\mathbb{R}^n)$.
	\end{theorem}
	
	In addition to the above theorem, through the variational method, we demonstrate in Theorem \ref{convergence-theorem} that $\ff$ asymptotically converges to the unique ground state of  
	\[
	\omega\ff_0-\Delta\ff_0=\ff_0^3
	\]
	when the convolution coefficient $\varepsilon$ tends to zero.
	
\subsection*{Some notations and tools}
	The local well-posedness result outlined in Theorem \ref{thmlwp} can be established via the standard contraction argument, using the
	the Bourgain-type space $X^{s,b}_\beta$ with norm
	\begin{align*}
		\| u \|_{X^{s,b}_\beta} := \norm{| \xi |^{-\beta}  \angles{ \xi }^{s+\beta} \angles{ \tau- m(\xi)}^b \widetilde{u} (\xi, \tau)}_{L^2_{\tau, \xi}}.
	\end{align*}
	The restriction to the time slab $(0, T)\times \mathbb R^n$
of the Bourgain space, denoted by $X^{s, b}_{\beta, T}$,
is a Banach space when equipped with the norm
$$
\| u \|_{X^{s, b}_{\beta, T}}  = \inf \left\{ \| v \|_{X^{s, b}_{\beta}}: \ v = u \text{ on }  (0, T) \times \mathbb{R}^n \right\}.
$$
So the result in Theorem \ref{thmlwp}  follows from  the following trilinear estimate whose proof is given in the Appendix.
\begin{lemma}
		\label{lm-bilest}
		Let $1/2<b<1$ , $0 < T < 1$,  $0<\beta<n/2$ and $s\geq(n-1)/2$. Then,
		\begin{equation}
			\label{biest1}
			\norm{ u_1 \bar u_2 u_3}_{X_{\beta, T}^{s, b-1}}  
			\lesssim   T^{1-b}
			\ \prod_{j=1}^3
			\norm{u_j }_{ X_{\beta, T}^{s, b}  }.
		\end{equation}
		
	\end{lemma}
One crucial tool in the proof of  Lemma \ref{lm-bilest} is a frequency localized  $L^1_x-L^{\infty}_x$ decay estimate for the linear propagator associated to \eqref{rnls}, viz.
	$
	\ee^{ \ii t m(D)},
	$
	where the symbol $D$ presents the operator $-\ii\nabla$. The frequency localization is defined as follows.  Fix an even smooth function $\chi \in C_0^{\infty}(\mathbb R)$ such that
	\begin{equation*}
		0 \le \chi \le 1, \quad
		\chi_{|_{[-1,1]}}=1 \quad \mbox{and} \quad  \mbox{supp}(\chi)
		\subset [-2,2]
	\end{equation*}
	and set
	$$
	\rho(s)
	=\chi\left(s\right)-\chi \left(2s\right).
	$$
	For a dyadic number
	$\lambda \in  2^\Z$,  we set $\rho_{\lambda}(s):=\rho\left(s/\lambda\right)$, and thus $\supp \rho_\lambda= 
	\{ s\in \R: \lambda/ 2 \le |s| \le 2\lambda \}$. 
	Now define the frequency projection $P_\lambda$ via
	\begin{align*}
		\widehat{P_{\lambda} f}(\xi)  = \rho_\lambda(|\xi|)\widehat { f}(\xi) .
	\end{align*}
	We sometimes write $f_\lambda:=P_\lambda f $, so that
	\[ f=\sum_{\lambda  } f_\lambda ,\]
	where summations throughout the paper are done over dyadic numbers in $ 2^\Z$.

	The decay estimate is given as follows. The proof is given in the Appendix below.
	\begin{lemma}  \label{lm-dispest}
		Let $m(r)=r^2+\e  r^{-2\beta}$ with $0<\beta <1$ and $\varepsilon =\pm 1$.
		
		Then, for $\lambda\gg 1$,
		\begin{align}
			\label{dispest}
			\|  \ee^{\ii tm(D)} P_{\lambda} f  \|_{L^\infty_x(\R^n)} &\le C  |t|^{-\frac n2}  \| f\|_{L_x^1(\R^n)} \end{align}
		for all $f \in \mathcal{S}(\R^n)$.
	\end{lemma}

	Combining the aforementioned dispersive estimate with the standard $TT^*$ argument, we can derive the following localized Strichartz estimates.
	\begin{lemma} \label{lm-LocStr}
		Let $\lambda \gg 1$ and 
		$(q, r)$ be \emph{Schr\"odinger admissible pair} in the sense that
		\begin{equation} \label{admissible}
			q> 2, \ r\ge 2 \quad \text{and} \quad \frac2q + \frac nr=\frac n2 \, .
		\end{equation} 
		Then,
		\begin{align}
			\label{Strest1d}
			\norm{ \ee^{\ii tm(D)} f_{\lambda}}_{ L^{q}_{t} L^{r}_{ x} (\R^{n+1}) } \lesssim 
			\norm{  f_{\lambda}}_{ L^2_{ x}(\R^n )} ,
		\end{align}
		for all  $f \in \mathcal{S}(\R^n)$.
		Moreover, if $b>1/2$, we have by the transference principle,
		\begin{equation}
			\label{Str-transfer}
			\norm{ u_\lambda }_{L^q_{t} L^r_{x}  (\R^{n+1}) } \lesssim 
			\norm{u_\lambda}_{  X_\beta^{0, b} }.
		\end{equation}
		
	\end{lemma}

\vspace{3mm}

The rest of the paper is organized as follows. In the next section we prove some nonlinear estimates and then give the proof for Lemma \ref{lm-bilest}.
The last section is dedicated to investigating the existence and properties of standing waves, as well as the global boundedness of solutions in the energy space.

	\section{Nonlinear estimates and Proof of Lemma \ref{lm-bilest}}\label{sec-3}

	\subsection{Nonlinear estimates}
	\begin{lemma}
		\label{L3lemma}
		Let $1/2<b<1$ and  $0 < T < 1$. Assume that $\lambda_{\text{min}}$,  $\lambda_{\text{med}}$ and  $\lambda_{\text{max}}$
		denote the minimum, median and maximum of $\lambda_1$,  $\lambda_2$ and  $\lambda_3$, respectively.
		Then
		\begin{equation}\label{Trest}
			\norm{ P_{\lambda_4} \left(P_{\lambda_1}  u_1 P_{\lambda_2} u_2 P_{\lambda_3} u_3\right) }_{L_T^2 L^2_x}  
			\lesssim    B(\lambda)\ \prod_{j=1}^3
			\norm{ P_{\lambda_j} u_j }_{ X_0^{0, b}  },
		\end{equation}
		where 
		\begin{equation}
			\label{B1}
			B(\lambda) \sim [ \min( \lambda_{\min} \lambda_{\text{med}} , \lambda_{\min} \lambda_4)]^\frac n2.
		\end{equation}
		Moreover,  if $\lambda_{\text{med}} \gg 1 $ we can take 
		\begin{equation}
			\label{B2}
			B(\lambda) \sim 
			\begin{cases}
				1 \qquad  & \text{if} \quad   n=1,
				\\
				\min \left( \lambda_{\min}^\frac n2 \lambda_{\text{med}}^{\frac n 2-1} , \lambda_{\min} ^{\frac n2-1}\lambda_4^\frac n2 \right)    \qquad & \text{if} \quad   n\ge 2.
			\end{cases}
		\end{equation}
		
	\end{lemma}
	
	\begin{proof}
		By symmetry, we may assume 
		$ \lambda_1 \le \lambda_2 \le \lambda_3$.

		First, we prove \eqref{Trest}  with \eqref{B1}.
		If $\lambda_1 \le \lambda_4$, by the H\"{o}lder   and Sobolev inequalities,
		\begin{align*}
			\norm{ P_{\lambda_4} \left(P_{\lambda_1}  u_1 P_{\lambda_2} u_2  P_{\lambda_3} u_3 \right)}_{L_T^2L^2_x}  &\le T^{\frac 12} 
			\norm{P_{\lambda_1}  u_1}_{ L_T^\infty L_x^\infty  }  \norm{P_{\lambda_2} u_2}_{ L_T^\infty L_x^\infty  } 
			\norm{ P_{\lambda_3} u_3 }_{ L_T^\infty L_x^2  }
			\\
			& \lesssim T^{\frac 12} ( \lambda_1  \lambda_2)^\frac n2
			\prod_{j=1}^3
			\norm{ P_{\lambda_j} u_j }_{ X_0^{0, b}  }.
		\end{align*}
		On the other hand, if $\lambda_4 \le \lambda_1$, then
		\begin{align*}
			\norm{ P_{\lambda_4} \left(P_{\lambda_1}  u_1 P_{\lambda_2} u_2  P_{\lambda_3} u_3 \right) }_{L_T^2L^2_x} &\lesssim T^{\frac 12}  \lambda_4^\frac n2 \norm{ P_{\lambda_1}  u_1 P_{\lambda_2} u_2  P_{\lambda_3} u_3  }_{L_T^\infty L^1_x}
			\\
			& \lesssim  T^{\frac 12}  (\lambda_1\lambda_4)^\frac n2 \norm{P_{\lambda_1}  u_1}_{ L_T^\infty L_x^\infty  } 
			\norm{P_{\lambda_2} u_2 }_{ L_T^\infty L_x^2 } 
			\norm{ P_{\lambda_3} u_3 }_{ L_T^\infty L_x^2  }
			\\
			& \lesssim T^{\frac 12} ( \lambda_1 \lambda_4)^\frac n2
			\prod_{j=1}^3
			\norm{ P_{\lambda_j} u_j }_{ X_0^{0, b}  }.
		\end{align*}

		Next, we prove \eqref{Trest}  with \eqref{B2} in the case $\lambda_2 \gg 1$.
		
		\textbf{ Case $n=1:$}
		By the
		H\"{o}lder inequality,  Bernstein inequality and then Lemma \ref{lm-LocStr}, we get 
		\begin{align*}
			\norm{ P_{\lambda_4} \left(P_{\lambda_1}  u_1 P_{\lambda_2} u_2  P_{\lambda_3} u_3 \right) }_{L_T^2L^2_x} &\le    \norm{ P_{\lambda_1}  u_1}_{ L^\infty_T L^{2}_x  }
			\norm{P_{\lambda_2}  u_2}_{ L_T^4 L_x^{ \infty } } \norm{P_{\lambda_3}  u_3}_{ L_T^4 L_x^{ \infty } }
			\\
			&\lesssim   \prod_{j=1}^3
			\norm{ P_{\lambda_j} u_j }_{ X_0^{0, b}  }.
		\end{align*}
		
		\textbf{ Case $n\ge 2:$}
		If
		$\lambda_1 \le \lambda_4$, we apply the H\"{o}lder inequality,  Bernstein inequality, and finally Lemma \ref{lm-LocStr}, to obtain 
		\begin{align*}
			\norm{ P_{\lambda_4} \left(P_{\lambda_1}  u_1 P_{\lambda_2} u_2  P_{\lambda_3} u_3 \right) }_{L_T^2L^2_x} &\le    \norm{ P_{\lambda_1}  u_1}_{ L^\infty_T L^{\infty}_x  }
			\norm{P_{\lambda_2}  u_2}_{ L_T^4 L_x^{2n} }  \norm{P_{\lambda_3}  u_3}_{ L_T^4 L_x^{\frac {2n} {n-1}} } 
			\\
			&\lesssim   \lambda_1^\frac n2  \lambda_2^{\frac n2 -1} \norm{ P_{\lambda_1}  u_1}_{ L^\infty_T L^{2}_x  }
			\norm{P_{\lambda_2}  u_2}_{ L_T^4 L_x^{\frac {2n} {n-1}} }  \norm{P_{\lambda_3}  u_3}_{ L_T^4 L_x^{\frac {2n} {n-1}} } 
			\\
			&\lesssim  \lambda_1^\frac n2  \lambda_2^{\frac n2 -1}     \ \prod_{j=1}^3
			\norm{ P_{\lambda_j} u_j }_{ X_0^{0, b}  }.
		\end{align*}

		Similarly, if
		$\lambda_4 \le \lambda_1$,
		\begin{align*}
			\norm{ P_{\lambda_4} \left(P_{\lambda_1}  u_1 P_{\lambda_2} u_2  P_{\lambda_3} u_3 \right) }_{L_T^2L^2_x} &\lesssim \lambda_4^\frac n2 \norm{ P_{\lambda_1}  u_1 P_{\lambda_2} u_2  P_{\lambda_3} u_3  }_{L_T^2 L^1_x}
			\\
			& \lesssim  \lambda_4^\frac n2   \norm{P_{\lambda_1}  u_1}_{ L_T^\infty L_x^n }  
			\norm{P_{\lambda_2}  u_2}_{ L_T^4 L_x^{\frac {2n} {n-1}} }  \norm{P_{\lambda_3}  u_3}_{ L_T^4 L_x^{\frac {2n} {n-1}} } 
			\\
			& \lesssim \lambda_4^\frac n2  \lambda_1^{\frac n2 -1}  
			\prod_{j=1}^3
			\norm{ P_{\lambda_j} u_j }_{ X_0^{0, b}  }.
		\end{align*}

	\end{proof}

	\begin{lemma}
		\label{L3lemma}
		For $1\ll\mu \ll \lambda$, we have
		\begin{equation}\label{Trest1}
			\norm{ \ee^{\ii tm(D)}P_{\mu}   f \cdot    \ee^{\ii t tm(D)}P_{\lambda} g }_{L_{t,x}^2 ( \R^{n+1}) }  
			\lesssim   \mu^{\frac{n-1}{2}} \lambda^{-\frac12 }
			\norm{ f }_{ L_x^2(\R^n)    }  \norm{ g }_{ L_x^2(\R^n) }.
		\end{equation}
		
		Moreover, 
		\begin{equation}\label{TTrest}
			\norm{ P_{\mu} u  \cdot   P_{\lambda} v }_{L_{t,x}^2 ( \R^{n+1}) }  
			\lesssim   \mu^{\frac{n-1}{2}} \lambda^{-\frac12 }
			\norm{ P_{\mu} u}_{ X^{0, b}_0   }  \norm{ P_{\lambda} v }_{ X^{0, b}_0 }.
		\end{equation}
		
	\end{lemma}
	\begin{proof}
		By duality,
		it suffices to show 
		\begin{equation}
			\label{dual}
			\int_{\R^n}  \int_{\R^n} F_\lambda(\xi + \eta , m(\xi) + m(\eta ))
			\widehat{f_\mu } (\xi ) \widehat{g_\lambda} (\eta )  
			\dd\xi \dd\eta
			\lesssim   \mu^{\frac{n-1}{2}} \lambda^{-\frac12 } {{\| f \|}_{L^2(\R^{n})} } {{\| g
					\|}_{L^2(\R^{n})} }   {{\|  F\|}_{L^2(\R^{n+1})} }.
		\end{equation}
		
		By renaming components, we may assume that $|\xi_1 | \thicksim
		|\xi | \sim \mu$ and $|\eta_1 | \thicksim |\eta |\sim \lambda$. Write $\xi = (\xi_1
		, \overline \xi ).$ We now change variables by writing $\sigma= \xi +
		\eta,~ \alpha=m(\xi) +
		m(\eta)$ and $\dd\sigma\dd\alpha= J \dd\xi_1 \dd\eta$. A calculation
		then shows that 
		$$J = \left|2 (\xi_1 - \eta_1) - \beta\varepsilon \xi_1|\xi|^{-2\beta-2} +  \beta \varepsilon \eta_1|\eta|^{-2\beta-2} \right| \sim \lambda.$$
		Therefore, upon changing variables in the inner two integrals, we
		have
		\begin{equation*}
			\text{LHS} \ \eqref{dual}= \int_{\R^{n-1}}  \left[\int_{\R} \int_{\R^n}     F_\lambda (\sigma, \alpha)    \cdot   G_{\mu
				,\lambda} (  \sigma, \alpha, \overline{\xi})  \dd\sigma \dd\alpha  \right]\dd \overline\xi,
		\end{equation*}
		where
		\begin{equation*}
			G_{\mu, 
				\lambda} ( \sigma , \alpha, \overline{\xi}) = \frac{    \widehat{f_\mu } (\xi ) \widehat{g_\lambda} (\eta )    }{J}.
		\end{equation*}
		
		We apply the Cauchy-Schwarz inequality on the $\sigma, \alpha$ integration and change back to
		the original variables to obtain
		\begin{align*}
			\text{LHS} \ \eqref{dual} &= \| F_\lambda \|_{L^2(\R^{n+1})}  \cdot   \int_{\R^{n-1}}  \left[\int_{\R} \int_{\R^n}         \abso{\widehat{f_\mu } (\xi ) \widehat{g_\lambda} (\eta )  }^2  J^{-2} \dd\sigma \dd\alpha \right]^\frac12 \dd \overline{\xi}
			\\
			& =  \| F_\lambda \|_{L^2(\R^{n+1})}  \cdot   \mu^{\frac{n-1}2} \left[   \int_{\R^{n-1}}  \int_{\R} \int_{\R^n}       \abso{ \widehat{f_\mu } (\xi ) \widehat{g_\lambda} (\eta )  }^2  J^{-2} \dd\sigma \dd\alpha  \dd \overline{\xi} \right]^\frac12 
			\\
			& =  \| F_\lambda \|_{L^2(\R^{n+1})}  \cdot   \mu^{\frac{n-1}2} \left[   \int_{\R^{n}} \int_{\R^n}       \abso{ \widehat{f_\mu } (\xi ) \widehat{g_\lambda} (\eta )  }^2  \lambda^{-1} \dd\xi \dd\eta  \right]^\frac12 
			\\
			& \le  \mu^{\frac{n-1}2} \lambda^{-\frac12 } {{\| f \|}_{L^2(\R^{n})} } {{\| g
					\|}_{L^2(\R^{n})} }   {{\|  F\|}_{L^2(\R^{n+1})} } .
		\end{align*}

	\end{proof}

\begin{lemma}
\label{LL3lemma}
For $1\ll\lambda_1 \ll \lambda_2$, we have
\begin{equation}\label{TT3rest}
\norm{ P_{\lambda_4} \left(P_{\lambda_1}  u_1 P_{\lambda_2} u_2 P_{\lambda_3} u_3\right) }_{ X_0^{0, \frac14}  }  
\lesssim  B(\lambda) \ \prod_{j=1}^3
  \norm{ P_{\lambda_j} u_j }_{ X_0^{0, b}  },
  \end{equation}
where 
\begin{equation}
\label{B2'}
B(\lambda) =  \min \left(\lambda_1^\frac {n-1}2  \lambda_2^\frac {n-2}2, \,  \lambda_1^\frac {n-1}2  \lambda_2^{-\frac 12}   \lambda_4^{\frac n2} \right) .
\end{equation}
\end{lemma}
	 \begin{proof}
	By duality,
it suffices to show 
\begin{equation}\label{T3rest}
\Bigabs{ \int P_{\lambda_1}  u_1 P_{\lambda_2} u_2 P_{\lambda_3} u_3 P_{\lambda_4}  u_4 \, dx dt }
\lesssim    \lambda_1^\frac {n-1}2  \lambda_2^\frac {n-2}2\ \prod_{j=1}^3
  \norm{ P_{\lambda_j} u_j }_{ X_0^{0, b}  } \norm{P_{\lambda_4}  u_4}_{ X_0^{0, \frac14}  }  ,
  \end{equation}
  
   By H\"{o}lder, Bernstein inequality and Lemma \ref{L3lemma},
   \begin{align*}
\text{LHS} \ \eqref{T3rest}  &\lesssim \norm{ P_{\lambda_1}  u_1 P_{\lambda_2} u_2 }_{L_t^2L_x^{2}  } \norm{  P_{\lambda_3} u_3 }_{ L_t^4 L_x^{2  }}  \norm{  P_{\lambda_4} u_3 }_{ L_t^4 L_x^\infty }
\\
&\lesssim \lambda_1^\frac {n-1}2  \lambda_2^{-\frac 12}   \lambda_4^{\frac n2} \prod_{j=1}^3
  \norm{ P_{\lambda_j} u_j }_{ X_0^{0, b}  } \norm{P_{\lambda_4}  u_4}_{ X_0^{0, \frac14}  } 
  \end{align*}

 On the other hand,  by H\"{o}lder, Bernstein inequality, Lemma \ref{L3lemma} and Lemma \ref{lm-LocStr},
  \begin{align*}
\text{LHS} \ \eqref{T3rest}  &\lesssim \sum_{\mu}\norm{ P_\mu \left( P_{\lambda_1}  u_1 P_{\lambda_2} u_2 \right) }_{L_t^2L_x^{2n}  } \norm{ P_\mu  \left( P_{\lambda_3} u_3 P_{\lambda_4}  u_4 \right) }_{{L_t^2 L_x^{\frac {2n} {2n-1}}  }  }
\\
&\lesssim \sum_{\mu} \mu^\frac {n-1}2 \norm{ P_\mu \left( P_{\lambda_1}  u_1 P_{\lambda_2} u_2 \right) }_{L_t^2 L_x^{2} } \norm{  P_{\lambda_3} u_3 }_{ L_t^4 L_x^{\frac {2n} {n-1}}  }  \norm{  P_{\lambda_4} u_3 }_{ L_t^4 L_x^2 }
\\
&\lesssim \sum_{\mu} \mu^\frac {n-1}2   \lambda_1^\frac {n-1}2 \lambda_2^{-\frac12 } \prod_{j=1}^3
  \norm{ P_{\lambda_j} u_j }_{ X_0^{0, b}  } \norm{P_{\lambda_4}  u_4}_{ X_0^{0, \frac14}  } 
  \\
&\lesssim  \lambda_1^\frac {n-1}2  \lambda_2^\frac {n-2}2 \prod_{j=1}^3
  \norm{ P_{\lambda_j} u_j }_{ X_0^{0, b}  } \norm{P_{\lambda_4}  u_4}_{ X_0^{0, \frac14}  } 
  \end{align*}

\end{proof}

		\subsection{Proof of Lemma \ref{lm-bilest}}\label{sec-3}
	
	Now we are in a position to prove the trilinear estimate.
	We have  (see e.g. \cite[Lemma 2.11]{tao-book})
	\begin{align}
		\label{TFactor}
		\norm{u}_{X_{\beta, T}^{s, b-1}} &\le C
		T^{1-b}\norm{u}_{ L_T^2 H_\beta^{s}},
	\end{align}
	where $C$ is independent on $T$.
	So in view of \eqref{TFactor},   estimate \eqref{biest1} reduces to proving
	\begin{equation}\label{biest2}
		\norm{    u_1 u_2 u_3 }_{L_T^2 H_\beta^{s}}  
		\lesssim \prod_{j=1}^3
		\norm{u_j }_{ X_\beta^{s, b}  }.
	\end{equation}
	
	By duality \eqref{biest2} reduces further to
	\begin{equation}\label{duality-biest11}
		\left| \int_0^T \int_{\R^n}   |D|^{-\beta}   \angles{D}^{s+\beta}  \left(  \prod_{j=1}^3  |D|^{\beta}  \angles{D}^{-s-\beta}    u_j  \right)  u_4 \dd x \dd t \right| \lesssim  \prod_{j=1}^3
		\norm{u_j }_{ X_0^{0, b}  } 
		\norm{ u_4 }_{  L_T^2 L_x^2 }.
	\end{equation}
		Decomposing $u_j= \sum_{\lambda_j >0} P_{\lambda_j} u_j$, we have
	\begin{equation}\label{duality-biestdecomp-c}
		\begin{split}
			\text{LHS \eqref{duality-biest11} } \  \lesssim \sum_{\lambda_1, \lambda_2, \lambda_3, \lambda_4  } \left| \int_0^T \int_{\R^n}    |D|^{-\beta}   \angles{D}^{s+\beta} P_{\lambda_4} \left(  \prod_{j=1}^3  |D|^{\beta}  \angles{D}^{-s-\beta}    P_{\lambda_j} u_j  \right) P_{\lambda_4} u_4\dd x \dd t \right|.
		\end{split}
	\end{equation}
	Set \begin{align*}
		a_{\lambda_j}:&= \norm{P_{\lambda_j} u_j}_{  X_0^{0, b} }, 
		\qquad
		a_{\lambda_4}:=  \norm{P_{\lambda_4} u_4}_{L^2_{T,x}   }  \qquad (j=1, 2, 3).
	\end{align*}
	By the Cauchy-Schwarz inequality, Lemma \ref{L3lemma}  and Bernstein inequality,
	\begin{equation}\label{maindecomp-c}
		\begin{split}
			\text{LHS \eqref{duality-biestdecomp-c}} \  &\lesssim
			\sum_{\lambda_1, \lambda_2, \lambda_3, \lambda_4 }
			\norm{  |D|^{-\beta}  \angles{D}^{s+\beta} P_{\lambda_4} \left(  \prod_{j=1}^3  |D|^{\beta}  \angles{D}^{-s-\beta}    P_{\lambda_j} u_j    \right)}_{L_{T,x}^2} \norm{ P_{\lambda_4} u_4  }_{ L_{T,x}^2  }
			\\
			&\lesssim \underbrace{
				\sum_{
					\lambda_1, \lambda_2, \lambda_3, \lambda_4 	}   C(\lambda) 
				a_{\lambda_1} a_{\lambda_2} a_{\lambda_3} a_{\lambda_4}}_{:=I},
		\end{split}
	\end{equation}
	where 
	$$
	C(\lambda) = B(\lambda) \cdot  \lambda_4^{-\beta}  \angles{\lambda_4}^{s+\beta}  \prod_{j=1}^3 
	\lambda_j^{\beta}   \angles{\lambda_j}^{-s-\beta}.
	$$
		So it suffices to show that
	\begin{equation}
		\label{I}
		I\lesssim  \prod_{j=1}^4 \|(a_{\lambda_j})\|_{l^2_{\lambda_j}}.
	\end{equation}

	\vspace{2mm}
	By symmetry of our argument, we may assume 
	$ \lambda_1 \le \lambda_2 \le \lambda_3$. 
	
	If $\lambda_3 \lesssim 1$, we use \eqref{B1} to obtain
	$$
	C(\lambda) \lesssim  (\lambda_1\lambda_4)^\frac n2 \cdot \lambda_4^{-\beta} ( \lambda_1  \lambda_2 \lambda_3)^\beta .
	$$
	Then applying the Cauchy-Schwarz inequality in $\lambda_1, 
	\lambda_2, \lambda_3 $ and $\lambda_4$, to
	obtain
	\begin{align*}
		I &\lesssim 
		\sum_{    \lambda_1,  \lambda_2, \lambda_3, \lambda_4 \lesssim 1 }   \lambda_4^{\frac n2-\beta} ( \lambda_1  \lambda_2 \lambda_3)^\beta
		a_{\lambda_1} a_{\lambda_2} a_{\lambda_3} a_{\lambda_4} \lesssim\prod_{j=1}^4 \|(a_{\lambda_j})\|_{l^2_{\lambda_j}},
	\end{align*}
	provided   $\beta<n/2$.

	So, we may assume $\lambda_3 \gg 1$.  We consider three cases:
	\begin{enumerate}[(i)]
		\item $ \lambda_1 \sim \lambda_2 \sim \lambda_3$,
		
		\item $ \lambda_1 \le \lambda_2 \ll \lambda_3$,
		\item $ \lambda_1 \ll \lambda_2 \sim \lambda_3$.
	\end{enumerate}

	\subsection*{ \underline{(i): $ \lambda_1 \sim \lambda_2 \sim \lambda_3$}}
	If $\lambda_4 \lesssim 1$, we use \eqref{B2} to obtain
$$
C(\lambda)  \sim 
\begin{cases}
 \lambda_4^{\frac n2-\beta}
 \lambda_3^{ -3s }   \qquad  & \text{if} \quad   n=1,
 \\
\lambda_4^{\frac n2-\beta}
 \lambda_3^{ \frac n2  -1-3s }     \qquad & \text{if} \quad   n\ge 2.
 \end{cases}.
$$ Therefore,
\begin{align*}
I &\lesssim 
 \sum_{ \lambda_4 \lesssim 1, \   \lambda_1\sim  \lambda_2\sim\lambda_3\gg 1 }  C(\lambda)  
a_{\lambda_1} a_{\lambda_2} a_{\lambda_3} a_{\lambda_4} \lesssim\prod_{j=1}^4 \|(a_{\lambda_j})\|_{l^2_{\lambda_j}}
\end{align*}
provided $0<\beta<n/2$,  $s\ge 0$ when $n=1$ and $s\ge (n-2)/6 $ when $n\ge 2$. 

 If $\lambda_4 \gg 1$, we use \eqref{B2} to obtain
$$
C(\lambda)  \sim 
\begin{cases}
 \lambda_4^{s}
 \lambda_3^{  -3s }  \qquad  & \text{if} \quad   n=1,
 \\
 \lambda_4^{\frac n2 +s}
 \lambda_3^{ \frac n2 -1  -3s }     \qquad & \text{if} \quad   n\ge 2.
 \end{cases}
$$
 and hence
\begin{align*}
I  &\lesssim 
 \sum_{ \lambda_4 \gg 1, \   \lambda_1\sim  \lambda_2\sim\lambda_3\gg 1 } C(\lambda)  
a_{\lambda_1} a_{\lambda_2} a_{\lambda_3} a_{\lambda_4} \lesssim\prod_{j=1}^4 \|(a_{\lambda_j})\|_{l^2_{\lambda_j}}
\end{align*}
provided $s\ge (n-1)/2 $.

	\subsection*{ \underline{(ii): $\lambda_1 \le \lambda_2 \ll \lambda_3$ }} 
 
 This implies $\lambda_3 \sim \lambda_4$.

\begin{itemize}

\item
If $\lambda_1 \le \lambda_2 \lesssim 1$, we use \eqref{B1} to obtain
$
C(\lambda) \sim  \lambda_1^{\frac n2+\beta}  \lambda_2^{\frac n2+\beta},
$
and hence \eqref{I} holds for all $s$.

\item Assume $\lambda_1 \lesssim 1$, $\lambda_2 \gg 1$. If $n\ge 2$, we use \eqref{B2}  to obtain
$$
C(\lambda)  \sim 
 \lambda_1^{\frac n2+\beta}    \lambda_2^{ \frac n2 -1-s } 
$$
which can be used to  get \eqref{I} for $s>(n-2)/2$.
 If $n= 1$, we use \eqref{B2'}  
$$
C(\lambda)  \sim 
 \lambda_1^{\beta}
 \lambda_2^{  -1/2 } 
$$
which can be used to  get \eqref{I} for all $s\in \R$ and $\beta>0$.

\item If $ 1 \ll \lambda_1 \sim \lambda_2$, we use \eqref{B2} to get
$$
C(\lambda)  \sim 
\begin{cases}
 \lambda_2^{  -2s }  \qquad  & \text{if} \quad   n=1,
 \\
  \lambda_2^{n-1-2s}    \qquad & \text{if} \quad   n\ge 2
 \end{cases}
$$
from which \eqref{I} follows 
provided $s\ge (n-1)/2$.

\item If $ 1 \ll \lambda_1 \ll \lambda_2$, we use \eqref{B2'}  to get
$$
C(\lambda)  \sim  \lambda_1^{\frac{n-1} 2-s} 
 \lambda_2^{\frac{n-2} 2-s}   
$$
from which \eqref{I} follows 
provided $s\ge (n-1)/2$.

		\item If $ 1 \ll \lambda_1 \sim \lambda_2$, we use \eqref{B2} to get
		$$
		C(\lambda)  \sim 
		\begin{cases}
			\lambda_2^{  -2s }  \qquad  & \text{if} \quad   n=1,
			\\
			\lambda_2^{n-1-2s}    \qquad & \text{if} \quad   n\ge 2
		\end{cases}
		$$
		from which \eqref{I} follows 
		provided $s\ge (n-1)/2$.

		\item If $ 1 \ll \lambda_1 \ll \lambda_2$, we use \eqref{B2'}  to get
		$$
		C(\lambda)  \sim  \lambda_1^{\frac{n-1} 2-s} 
		\lambda_2^{\frac{n-2} 2-s}   
		$$
		from which \eqref{I} follows 
		provided $s\ge (n-1)/2$.
		
	\end{itemize}

	\subsection*{ \underline{(iii): $\lambda_1 \ll \lambda_2 \sim \lambda_3$ }}

\begin{itemize}

\item If $\lambda_4,  \lambda_1  \lesssim 1$, we use \eqref{B1} to get $
C(\lambda) \sim \lambda_4^{\frac n2-\beta}   \lambda_1^{\frac n2+\beta} \lambda_2^{-2s } 
$, and therefore \eqref{I} follows for $0<\beta <n/2$,  $s\ge 0$.  

\item If $\lambda_4 \lesssim 1$ and $\lambda_1 \gg 1$, we use \eqref{B2'} to get
$$
C(\lambda)  \sim 
 \lambda_4^{\frac n2-\beta}  \lambda_1 ^{ \frac {n-1}2-s}  \lambda_2^{-\frac12-2s }    
$$
which can be used to  get \eqref{I}  provided $0<\beta <n/2$ and $s\ge (n-1)/2$.

\item If $\lambda_4 \gg 1$ and $\lambda_1 \lesssim 1$,  we use \eqref{B2} to obtain
$$
C(\lambda)  \sim 
\begin{cases}
 \lambda_1^{\beta} \lambda_4^{  s } 
 \lambda_2^{  -2s }  \qquad  & \text{if} \quad   n=1,
 \\
   \lambda_4^{ \frac n2 -1+s }  \lambda_1^{\frac n2+\beta}  \lambda_2^{ -2s }     \qquad & \text{if} \quad   n\ge 2
 \end{cases}
$$
which can be used to  get \eqref{I}  provided   $s\ge 0$ for $n=1$ and $s\ge (n-2)/2$ for $n\ge 2$.

\item Finally, if
$\lambda_4, \lambda_1 \gg  1$, we use \eqref{B2'} to obtain
$$
C(\lambda)  \sim  \lambda_4 ^s \lambda_1^{\frac{n-1} 2-s} 
 \lambda_2^{\frac{n-2} 2-2s}   =(\lambda_4/ \lambda_2)^s \lambda_1^{\frac{n-1} 2-s} 
 \lambda_2^{\frac{n-2} 2-s}   
$$
from which \eqref{I} follows 
provided $s\ge (n-1)/2$.

\end{itemize}

	\section{Global well-posedness}\label{sec-4}
	In this section, we aim to identify the conditions under which the Cauchy problem \eqref{rnls} exhibits global well-posedness. Building upon the local well-posedness result in Theorem \ref{thmlwp}, the following outcome readily follows from the conservation law $\E$. Throughout this section, we need to assume that $\varepsilon=+1$ because the term $\scal{u, L_\beta u}_\lt$ cannot be controlled if $\varepsilon$ is negative.
		In the focusing case $\varsigma=+1$, the energy $\E$ may be sign-changing, so it is classical to estimate the last term of $\E$ associated to the nonlinearity of \eqref{rnls}. This is strongly related to the best constant of embedding of $X_\beta$ in $L^4(\rn)$.
	\subsection{Standing waves}
	To establish global conditions, we first derive the following embedding.
	\begin{lemma}\label{embeding-lemma}
		There exists a constant $\vr=\vr(p,\beta,n)$
		such that
		\begin{equation}\label{embed-1}
			\|u\|_{L^{p+2}(\rn)}
			\leq \vr
			\|u\|_\lt^{\frac{4+p(2-n)}{2(p+2)}-\kappa}
			\|\nabla u\|_\lt^{\frac{np}{2(p+2)}+\frac{\kappa\beta}{\beta+1}}
			\|L_{\beta/2}u\|_\lt^{\frac\kappa{\beta+1}}
		\end{equation}
		holds for every function $u\in X_\beta$, where $0\leq\kappa\leq\frac{4+p(2-n)}{2(p+2)}$, $0< p\leq 2^\ast$ and
		\[
		2^\ast=\begin{cases}
			\infty^-,&n\leq2,\\
			\frac{4}{n-2},&n>2.
		\end{cases}
		\]
		Particularly, the space $X_\beta$ is continuously embedded into $L^{p+2}(\rn)$. Furthermore, $\x$ is embedded compactly into $L^p_{\rm loc}(\rn)$ for any $p\geq1$.
	\end{lemma}
	\begin{proof}
		Inequality \eqref{embed-1} is obtained from a simple combination of the following Gagliardo-Nirenberg inequalities (see for example \cite{amin-zamp}):
		\begin{equation}\label{gag-l}
			\|u\|_{L^{p+2}(\rn)}^{p+2}\leq
			\varrho_0
			\|u\|_\lt^{2+\frac{p}{2}(2-n)}
			\|\nabla u\|_\lt^{\frac{np}{2}},
		\end{equation}
		\blue{\begin{equation}\label{gag-2}
			\|(-\Delta)^{\frac{\beta}{2}}u\|_\lt
			\leq\|u\|_\lt^{\frac{1}{\beta+1}}\|(-\Delta)^{\frac{\beta+1}{2}}u\|_\lt^{\frac{\beta}{\beta+1}},
			\qquad\forall\beta\geq0,
		\end{equation}}
		where 	\begin{equation}\label{without-r}
			\varrho_0=\frac{2(p+2)}{np}\left(\frac{2(p+2)}{np}-1\right)^{\frac{np}{4}-1}
			\|\psi\|_\lt^{-p}
		\end{equation}
		and $\psi$ is the unique  ground state of \eqref{zero-beta-ground}.
			Then we deduce the continuous embedding $\x\hookrightarrow L^{p+2}(\rn)$. The compact embedding $L^q_{\rm loc}(\rn)$ for $q\in(2,2^\ast)$ is a direct consequence of the above continuous embedding by standard argument.
		
		To finish the proof of this lemma, it suffices to show that $\x$ is embedded compactly into $L^1_{\rm loc}(\rn)$. Let $\{u_k\}$ be a bounded sequence in $\x$. If we define $v_k=\chi u_k$, where $\chi\in C^\infty_0(\rn)$ such that $\chi\equiv1$ on $B_r(0)$ and 
		${\rm supp}(\chi)\subset B_{2r}(0)$, then $\{v_k\}$ is also a bounded sequence in $H^1(\rn)$. We can also assume that there is $v\in\lt$ such that  $v_k\to v$ in $\lt$ and ${\rm supp}(v)\subset B_{2r}(0)$. Now, we have for $R>0$ that $$\|\hat{v}_k -\hat{v}\|_{L^2(\rn\setminus B_{2r}(0))}^2\lesssim(1+R^2)^{-1},$$ and since $\exp(\ii x\cdot\xi)\in L^2(B_{2r}(0))$, we obtain from the weak convergence in $L(B_{2r}(0))$ that $\hat{v}_k$ converges  a.e. to $\hat{v}$. On the other hand, since
		\[
		\|\hat{v}_k\|_{L^\infty(\rn)}\leq\|v_k\|_{L^1(B_{2r}(0))}\lesssim_r\|v_k\|_{L^2(B_{2r}(0))}\lesssim_r\|v_k\|_\x,
		\]
		thus $|\hat{v}_k-\hat{v}|^2$ is bounded uniformly, and thereby the dominated convergence theorem implies that $$\|\hat{v}_k-\hat{v}\|_{L^2(B_R(0))}=o(1).$$ The above estimates combined with the Plancherel identity show that $v_k-v\to0$ in $L^2(\rn)$ and, hence $\|v_n-v\|_{L^1(B_{r}(0))}\to0$  by the H\"{o}lder inequality.
	\end{proof}	 
	
	Define the homogeneous space $\xx$ via
	\[
	\|u\|_\xx^2=\|\nabla u\|_{\lt}^2+ \|L_{\beta/2}u\|_\lt^2.
	\]
	Notice from the sharp inequalities
	\begin{equation}\label{norm-equi}
		\left(1+\frac{\beta^{\frac{\beta}{1+\beta}}}{1+\beta}\right)\|u\|_{\dot{X}_\beta}^2\geq\|u\|_\x^2\geq \|u\|_{\dot{X}_\beta}^2
	\end{equation}
	that $\|u\|_\x \stackrel{\beta}{\sim}\|u\|_{\dot{X}_\beta}$.
	\begin{lemma}\label{lemma-best-masszero}
		Let $n\geq1$. For any function $u\in X_\beta$, there holds that 
		\begin{equation}\label{best-constant-zero-mass}
			\|u\|_{L^{p+2}(\rn)}^{p+2}\leq\vr_\ast\|\nabla u\|_\lt^{\frac{\beta}{2(\beta+1)}(2p+4-np)+\frac{np}{2}}
			\norm{L_{\beta/2} u}_\lt^{\frac{1}{2(\beta+1)}(2p+4-np)},
		\end{equation}
		where
		\[
		\vr_\ast^{-1}=
		C_{\beta,n,p}
		\|Q_\ast\|_{\dot{X}_\beta}^{p},
		\]
		\[
		\begin{split}
			C_{\beta,n,p}&=\left(\frac{1}{2(1+\beta)(2+p)}\right)^{\frac{p+2}{2}}
			(p(n+2\beta)+4\beta)^{\frac{\beta}{4(\beta+1)}(2p+4-np)+\frac{np}{4}} 
			(4-p(n-2))^{\frac{1}{4(\beta+1)}(2p+4-np)}\\&<1
		\end{split}
		\]
		and $Q_\ast$ is a ground state of 
		\begin{equation}\label{gs-mass-zero}
			L_\beta\ff-\Delta\ff=\ff^{p+1}.
		\end{equation}	
		
	\end{lemma}

	\begin{proof}
		The proof is similar to one of Theorem 4.6 in \cite{wang-amin} with some modifications, so we omit the details.
	\end{proof}
	
	\begin{remark}
		It was shown in \cite[Proposition 5.1]{morpi} that \eqref{gag-2} is sharp. Hence, by combining \eqref{gag-l} and \eqref{gag-2}, we obtain from  \eqref{embed-1}  and \eqref{best-constant-zero-mass} that
		\[
\vr\leq\varrho_0\quad\text{and}\quad\varrho_\ast\leq\varrho_0.
		\]
	\end{remark}
	
	Before finding the conditions of guaranteeing the existence of global solutions, it is of independent interest to study the existence and properties of \eqref{gs-mass-zero}. Equation \eqref{gs-mass-zero} is a particular case of 
	\begin{equation}\label{standing}
		\omega \ff-\Delta\ff+\e  L_\beta\ff=\ff^{3}.
	\end{equation}
	Indeed, \eqref{standing} is deduced when one looks for a standing wave $u(x,t)=\ee^{\ii\omega t}\ff(x)$ of \eqref{rnls}.

	\begin{theorem}
		Let $\varepsilon>0$, $\beta>0$ and $\omega>\beta_\ast=-\left( \e\beta^{-\beta}\right)^{\frac{1}{\beta+1}}\left(1+\beta\right)$. Then equation \eqref{standing} possesses a nontrivial solution $\ff\in X_\beta$. Moreover, there exists a ground state $u$ of \eqref{standing}. Furthermore, $u\in\Gamma$ and $\tilde d_\omega=d_\omega$, where
		\[
		\tilde d_\omega=S_\omega(u)=\inf_{v\in\Gamma}S_\omega(v),
		\qquad
		d_\omega=\inf_{\gamma\in\Gamma}\max_{t\in[0,1]}S_\omega(\gamma(t)),
		\]
		and $\Gamma=\sett{\gamma\in C([0,1],X_\beta),\;\gamma(0)=0,S_\omega(\gamma(1))<0}$ and $\tilde\Gamma=\{\ff\in X_\beta\setminus\{0\},\;I(\ff)=\scal{S'(\ff),\ff}=0\}$.
	\end{theorem}
	\begin{proof}
		The proof follows a standard approach utilizing the Mountain-Pass Lemma without requiring the Palais-Smale condition. For the sake of completeness, we provide a sketch. Indeed,  define
		\[
		S_\omega(\ff)=\frac12\int_\rn\paar{\abso{L_{\beta/2}\ff}^2+|\nabla\ff|^2+\omega|\ff|^2}\dx-\frac14\|\ff\|_{L^4(\rn)}^4,
		\]
		and see that the critical points are the weak solutions of \eqref{standing}. By using \eqref{embed-1} and the fact
		\[		\int_\rn\paar{\abso{L_{\beta/2}\ff}^2+|\nabla\ff|^2+\omega|\ff|^2}\dx\gtrsim\|\ff\|_{X_\beta}^2
		\]
		for any $\omega>\beta_\ast$, we observe that
		\[
		S_\omega(\ff)\gtrsim\frac14\|\ff\|_{X_\beta}^2-
		\|\ff\|_{X_\beta}^4\geq\delta>0,\qquad \ff\in B_r(0)\subset X_\beta
		\]
		for some $\delta>0$, independent of $\ff$, and $r>0$ small enough. The structure of $S_\omega$ shows that the existence of $e\in B_r^c(0)\subset X_\beta$ with $S_\omega(e)_{X_\beta}<0$. On the other hand, \eqref{embed-1} shows that any mountain-pass sequence is bounded in $X_\beta$. Finally, the Mountain-pass geometry implies the existence of a nontrivial solution of \eqref{standing} by using the local compactness in Lemma \ref{embeding-lemma} and the non-vanishing property in \cite[Lemma 2.14]{amin-pas}. The second part of the proof comes from the same lines of Theorems 2.20 and 2.21 in \cite{amin-pas}, and we omit the details.
	\end{proof}

	In the case $n\geq2$, we can benefit from the compact embedding $X_\beta ^{\rm rad} \hookrightarrow L^p(\rn)$ for any $p$ given in Lemma \ref{embeding-lemma} and show the existence of radially symmetric solutions of \eqref{standing}.
	
	\begin{theorem}\label{radial-sol-thm}
		Let $n\geq2$,	$\beta>0$ and $\omega>\beta_\ast$. Then   any minimizing sequence of
		\[
		\inf\sett{I_\e(u),\;u\in X_\beta^{\rm rad},\|u\|_{L^4(\rn)}^4=\lam},
		\]
		is relatively compact in $X_\beta^{\rm rad}$. Furthermore, 
		\[
		\lim_k\inf_{\ff\in\ngg}\inf_{y\in\rn}\|\psi_k(\cdot+y)-\ff\|_\x=0 \quad\mbox{and}\quad
		\lim_k\inf_{\ff\in\ngg} \|\psi_k-\ff\|_\x=0,
		\]
		where $\ngg$ is the set of all ground states of \eqref{standing}.
	\end{theorem}	
	\begin{remark}
		It is observed from \eqref{standing} that no solution $\ff$ can be positive and
  \[
  \int_\rn\ff\;\dd x=0.
  \] This is contrary to the case $\varepsilon=0$ where the ground state is positive.
	\end{remark}
	
	Solutions of \eqref{standing}, obtained as rescaled minimizers of the above minimization problem, are henceforth referred to as ground-state solutions. We note that the existence of the (rescaled) ground states of \eqref{standing}, up to a rescaling, is equivalent to the existence of minimizers of
	\begin{equation}\label{m-min}
		m_\e=\inf_{u\in\x\setminus\{0\}}\frac{  I_\e(u)}{\|u\|_{L^{4}(\rn)}^2},
	\end{equation}
	where $I_\epsilon(u)=2S_\omega(u)+2\|u\|_{L^4}^4$.
	
	We delve into the behavior of the minimizers of \eqref{m-min} as $\e$ varies. It is recalled from \cite{lions} that the ground states $\ee^{\ii\omega t}\ff(x)$ of \eqref{nls} are related to the minimizers of
	\begin{equation}\label{zer0-m-min}
		m_0=\inf_{u\in H^1(\rn)\setminus\{0\}}\frac{ I_0(u)}{\|u\|_{L^{4}(\rn)}^2}.
	\end{equation}
	
	To emphasize the dependence of $m$ on $\e$, we denote $\ngg$ by $\ngg_\e$. It is well-known that the ground states of \eqref{nls}, up to translation, are unique, such that $\ngg_0$ consists of $\ff_0(\cdot+y)$ for all $y\in\rn$ and some smooth function $\ff_0\in H^1(\rn)$. In the one-dimensional case $n=1$, it takes the explicit form
	\[
	\ff_0(x)=\sqrt{2\omega}\;{\rm sech} \left( \sqrt{\omega}x\right).
	\]
	\begin{lemma}\label{density-space}
		The space $\x$	 is dense in $H^1(\rn)$.
	\end{lemma}
	
	\begin{proof}
		Let $u\in H^1(\rn)$ and $\delta>0$. Define the function $u_\delta$ by 
		\[
		\hat{u}_\delta(\xi)=\hat{u}(\xi)\chi_\delta(\xi),
		\]
		where $\chi_\delta=\chi_{\{\rn\setminus B_\delta(0)\}}$. Then, $u_\delta\in H^1(\rn)$, and we have from the Plancherel theorem that
		\[
		\|L_\beta u\|_\lt^2=\int_{\rn\setminus B_\delta(0)}|\xi|^{-2\beta}|\hat{u}(\xi)|^2\dd\xi<\delta^{-2\beta}\|u\|_\lt^2<+\infty.
		\]
		Hence, $u_\delta\in\x$. On the other hand, 
		\[
		\|u-u_\delta\|_{H^1(\rn)}^2=\int_{B_\delta(0)}(1+|\xi|^2)|\hat{u}(\xi)|^2\dd\xi\leq\|u\|_{H^1(\rn)}^2<+\infty.
		\]
		Hence, $u_\delta\to u$ in $H^1(\rn)$ as $\delta\to0$; and the proof is complete.
	\end{proof}
	\begin{theorem}\label{convergence-theorem}
		Let $\omega,\beta>0$ be fixed. Assume that $\{\e_k\} \subset\rr^+$ is a sequence converging to zero. If $\ff\in\ngg_{\e_k}$, then there exists a subsequence of $\{\e_k\}$, denoted by the same $\e_k$, and $\{y_k\}\subset\rn$ such that $\ff_k(\cdot+y_k)$ converges in $H^1(\rn)$ to $\ff_0$ as $\e_k\to0$.
	\end{theorem}
	
	\begin{proof}
		If we show that $m_\e$ is continuous at $\e=0$, then 
		\[
		\lim_k\|\ff_k\|_{L^{4}(\rn)}^{4}=\lim_k m_{\e_k}^{2}=m_0^{2}
		\]
		and
		\[
		\limsup_kI_0(\ff_k)=\limsup_k\left(I_{\e_k}(\ff_k)-\e_k\|L_{\beta/2}\ff_k\|_\lt^2\right)
		\leq\lim_k I_{\e_k}(\ff_k)=
		\lim_k m_{\e_k}^2
		=m_0^2;
		\]
		and the proof of the theorem is complete. 
		
		To prove the continuity of $m_\e$ at $\e=0$, it is enough to prove that $m_{\e}\to m_0$ as $\e\to0^+$. Observe that $m_\e$ is an increasing and continuous function in $\e$. Indeed, if $\ff_1$ and $\ff_2$ are ground states with parameters $\e_1<\e_2$  in \eqref{standing}, respectively, then we have from the definition of $m_\e$ that
		\[
		\begin{split}
			m_{\e_1}&=\frac{ I_{\e_1}(\ff_1)}{\|\ff_1\|_{L^{4}(\rn)}^2} 
			\leq  
			\frac{ I_{\e_1}(\ff_2)}{\|\ff_2\|_{L^{4}(\rn)}^2} 
			\leq
			\frac{ I_{\e_2}(\ff_2)-(\e_2-\e_1)\|L_{\beta/2}\ff_2\|_\lt^2}{\|\ff_2\|_{L^{4}(\rn)}^2}\\
			&\leq m_{\e_2}-(\e_2-\e_1)
			\frac{\|L_{\beta/2}\ff_2\|_\lt^2}{\|\ff_2\|_{L^{4}(\rn)}^2} 
			<  m_{\e_2}.
		\end{split}
		\]
		Thus $m_\e$ is increasing in $\e$.	Similarly, we obtain that 
		\[
		m_{\e_2}\leq m_{\e_1}+
		(\e_2-\e_1)
		\frac{\|L_{\beta/2}\ff_1\|_\lt^2}{\|\ff_1\|_{L^{4}(\rn)}^2};
		\]
		and thereby
		\[
		|m_{\e_2}-m_{\e_1}|\lesssim m_{\e_1}(\e_2-\e_1).
		\]
		Thus $m_\e$ is locally Lipschitz continuous in $\e$. Now, by Lemma \ref{density-space}, we can choose $\phi_k\in\x$ such that $\|\phi_k-\ff_0\|_{H^1(\rn)}<\frac 1k$. By defining
		\[
		\al_k=\min\left\{k^{-1},\|L_{\beta/2}\phi_k\|_\lt^{-2}\right\},
		\]
		we have
		\[
		m_{\al_k}\leq\frac{I_{0}(\phi_k)+\al_k\|L_{\beta/2}\phi_k\|_\lt^2}{\|\phi_k\|_{L^{4}(\rn)}^2}
		\leq
		\frac{I_{0}(\phi_k)+\frac1k}{\|\phi_k\|_{L^{4}(\rn)}^2}.
		\]
		Taking $\limsup$, we get from the continuity that
		$\limsup_km_{\al_k}\leq m_0$. On the other hand, for any $u_k\in\ngg_{\al_k}$ we obtain that
		\[
		m_0\leq\frac{I_0(u_k)}{\|u_k\|_{L^{4}(\rn)}^2}
		\leq\frac{I_{\al_k}(u_k)-\al_k\|L_{\beta/2}u_k\|_\lt^2}{\|u_k\|_{L^{4}(\rn)}^2}
		<\frac{I_{\al_k}(u_k)}{\|u_k\|_{L^{4}(\rn)}^2}=m_{\al_k}.
		\]
		This means that $\liminf_km_{\al_k}\geq m_0$, and the claim follows.
	\end{proof} 
	\begin{lemma}\label{variational-min-1}
		Let $\omega>\beta_\ast$. Then $\ngg$ is nonempty and $\ff\in\ngg$ iff
		$N(\ff)=0$  and
		\[
		S(\ff)=\inf_{u\in N}S(u),
		\]
		where $N=\sett{u\in\x\setminus\{0\},\,N(u)=\langle S'(u),u\rangle_\lt=0}$.
	\end{lemma}
	\begin{proof}
		The proof is similar to one of Theorem 2.1 in \cite{amins-2022} with natural modifications. So we omit the details.
	\end{proof}
	\begin{corollary}\label{elimit}
		Let $\ff\in\ngg$ and $\omega>0$, then
		\[
		\lim_{\e\to0^+}	\e\|L_{\beta/2}\ff\|_\lt=0.
		\]
	\end{corollary} 
	\begin{proof}
		First we note from Lemma \ref{variational-min-1} that 
		\[
		\e\|L_{\beta/2}\ff\|_\lt^2=-\omega\|\ff\|_\lt^2
		-\|\nabla\ff\|_\lt^2+m^2.
		\]
		It transpires from Theorem \ref{convergence-theorem} that
		\[
		\lim_{\e\to0^+}	\e\|L_{\beta/2}\ff\|_\lt^2=
		-\omega\|\ff_0\|_\lt^2-\|\nabla\ff_0\|_\lt^2+m_0^2
		=-I_0(\ff_0)+m^2=0.
		\]
	\end{proof}
	
	We define the spaces
	\[
	X^s=\sett{u\in L^\infty(\rn),\;\|u\|_{X^s}=\|\langle\cdot\rangle^s u\|_{L^\infty(\rn)}<\infty},\quad
	\hat{X}^s=\sett{u\in \mathcal{S}'(\rn),\;\|u\|_{\hat{X}^s}=\| \hat{u}\|_{X^s}<\infty},
	\]
	where $\langle \cdot\rangle=1+|\cdot|$.

	\begin{proof}[Proof of Theorem \ref{decay-theor}]
		First, we note that if $n<4$, then $\ff\in L^\infty(\rn)$ due to the Sobolev embedding. For the case $n\geq4$, we rewrite \eqref{standing} by
		$\ff=K_\beta\ast\ff^{3}$ where
		\begin{equation}\label{integral-form}
			\widehat{K}_\beta(\xi)=\frac{|\xi|^{2\beta}}{\e+\omega|\xi|^{2\beta}+|\xi|^{2(1+\beta)}}
		\end{equation}
		It is straightforward to check that  $\hat{K}_\beta\in L^{q}(\rn)$ with $q>\frac{n}{2}$, hence $K_\beta\in L^q(\rn)$ with $\frac{1}{q}>1-\frac2n$. In particular, 
		$\hat{K}_\beta(0)=\int_\rn K(x)\dx=0$. Moreover, 
		$|\xi|^{-2\beta}\hat{K}_\beta\in L^{q}(\rn)$ with $q>\frac{n}{2(1+\beta)}$, so $K_\beta\in L^q(\rn)$ with $\frac{1}{q}>1-\frac{2(1+\beta)}{n}$. 
		In addition, for any $\al>0$,
		$|\xi|^{\al}\hat{K}_\beta\in L^{q}(\rn)$ with $(2-\al)q>n$, so that $D^\al K_\beta\in L^q(\rn)$ with $\frac{1}{q}>1-\frac{2-\al}{n}$. 
		Therefore, it is easy from \eqref{integral-form} that $\ff\in L^\infty(\rn)$.
		To see the high regularity of $\ff$, one can check that $L_\beta K_\beta$ satisfies the  Mikhlin theorem. Then it follows from the fact $\ff\in L^\infty(\rn)$ that $L_\beta\ff\in H^{3+2\beta}(\rn)$. Now by using the inequalities  $D^\beta|u|\leq {\rm sgn}(u)D^\beta u$  and 
		$D^\beta\phi(f)\leq\phi'(f)D^\beta f$ (see \cite{corm}), we obtain that $\ff\in H^{3}(\rn)$.  
		Finally,   the classical bootstrapping argument shows that 	$\ff ,L_\beta\ff  \in  H^\infty(\rn)$.

		Next, we obtain the decay estimate of $\ff$. 
		We have by using the Plancherel theorem that $|x|^\al K_\beta\in L^q(\rn)$ for all $2+\al>\frac nq$. This gives the decay estimate of $K_\beta$ on $\rn\setminus B_R(0)$ for some $R>0$.

		By the regularity of solitary waves, for any $\delta>0$ we can find $R_\delta>0$ such that $|\ff^{3}|\leq\delta|\ff|$ for all $|x|\geq R_\delta$. We notice that $|x|^\al\ff\in L^q_{\rm loc}(\rn)$, so we show that $|x|^\al\ff\in L^q(\rr^n_R)$, where $\rr^n_R=\rr^n\setminus B_{R}(0)$. Following \cite{esfahan-phys-a}, we define $g_\epsilon(x)=A_\epsilon(x)\ff(x)$ with $A_\epsilon(x)=\frac{|x|^\al}{(1+\epsilon|x|)^s}$, where $\epsilon>0$ is arbitrary. Then,  we get from
		\[
		|\ff(x)|\lesssim\|(1+|x|)^sK_\beta\|_{L^{q'}(\rn)}\|G(x,\cdot)\|_{L^q(\rn)},\qquad G(x,y)=\frac{\ff^{3}(x)}{(1+|x-y|)^s} 
		\] 
		for some $s-\al>\frac{n}{q}$, and the H\"{o}lder inequality
		\[
		\|g_\epsilon\|_{L^q(\rr^n_R)}^q
		\lesssim
		\int_{\rr^n_R}A_\epsilon^q(x)\|G(x,\cdot)\|_{L^q(\rn)}^q\dx.
		\]
		By applying Lemma 2.3 in \cite{esfahan-phys-a}, we derive
		\[
		\|g_\epsilon\|_{L^q(\rr^n_R)}^q
		\lesssim
		\int_{\rr^n_R}A_\epsilon^q|\ff|^{q(p+1)}\dx
		+\int_{B_R(0)}|\ff(y)|^{q(p+1)}  
		\int_{\rr^n_R}\frac{A_\epsilon^q(x)}{(1+|x-y|)^{sq}}\dx\dd y
		\lesssim 
		\delta^q\|g_\epsilon\|_{L^q(\rr^n_R)}^q+1 .
		\]
		Now, if we choose $\delta$ small enough, then $\|g_\epsilon\|_{L^q(\rr^n_R)}^q$ is bounded uniformly in $\epsilon$. Hence, we let $\epsilon\to0$, and apply Fatou's lemma to deduce that
		$|x|^\al\ff\in L^q(\rr^n_R)$, where $\al$ is sufficiently small. Finally, we conclude from the inequality 
		\[
		\langle x
		\rangle^\al|\ff|\leq
		(\langle\cdot\rangle^\al K_\beta)\ast\ff^{3}
		+ K_\beta\ast(\langle\cdot\rangle^\frac{\al}{p+1}\ff)^{3}
		\]
		that $\langle x
		\rangle^\al\ff\in L^\infty(\rn)$. 
		
		To show the optimal decay estimate of the solitary wave, first, we consider the case $\beta\notin\N$.
		Since we have proved that $\ff\in X^\al$, then $\ff^{3}\in X^{3\al}$, then $\ff\in X^{3\al}\subset X^{n+4\beta}$, provided we   show   that $K_\beta$ is a continuous Fourier multiplier on $X^s$ for any $0\leq s\leq n+4\beta$.
		Let $\chi$ be a smooth cutoff function such that $\chi\equiv1$ in $B_1(0)$. By defining  
		$h_2(\xi)=\hat{K}_\beta-h_1(\xi)$, we prove that $h_1$ and $h_2$ are continuous multipliers on $X^s$, where 
		\[
		h_1(\xi)=\frac{\chi(\xi)|\xi|^{2\beta}}{\e+\omega|\xi|^{2\beta}+|\xi|^{2(1+\beta)}}.
		\]
		Following the ideas of \cite{capni}, we choose that $\tilde\chi\in C_0^\infty(\rn)$ such that $\tilde\chi\equiv1$ on ${\rm supp}(\chi)$, then it follows easily  from  Theorems 7.1.16 and  7.1.18 in \cite{hormander} that $\tilde\chi(\omega|\xi|^{2\beta}+|\xi|^{2(1+\beta)})\in\hat{X}^{n+2+2\beta}$. Now, if we rewrite $h_1$ by
		\[
		\frac1\e\chi(\xi)|\xi|^{2\beta} \sum_{k=0}(-1)^k
		\left(\frac{\omega|\xi|^{2\beta}+|\xi|^{2(1+\beta)}}{\e}\right)^k,
		\]
		then the above series is convergent in $\hat{X}^{n+2+2\beta}$ because 
		$\hat{X}^{n+2+2\beta}$ is a Banach algebra, and thereby $h_1\in \hat{X}^{n+4\beta+2}$ from the weighted Young inequality \cite{gfwz}. On the other hand, we note from the presence of $\chi$ in $h_2$ that $|h_2(\xi)|\gtrsim\langle\xi\rangle^{-2}$ and $D^\al h_2(\xi)\lesssim\langle\xi\rangle^{2-\al}$. So that $h_2$ belongs to the H\"{o}rmander class $S^{-2}$ (see \cite{stein}) and $h_2^\vee$ is a locally integralable function. Moreover, it follows from \cite{taylor} that $|((h_2)^\vee)(x)|\lesssim|x|^{-\ell}$ for any $|x|\geq1$ and $\ell>0$, and $|((h_2)^\vee)(x)|\lesssim|x|^{2-n}$ for any $|x|\leq1$. Indeed, $(h_2)^\vee$ is in the Gelfand–Shilov space $S_1^1$ (see \cite{cck,rt}). Therefore, $h_2$ is a continuous multiplier on $\hat{X}^{n+4\beta+2}$. This completes the proof in this case. 
		When $\beta\in\N$,   equation \eqref{standing} can be written by 
		\[
		(\e+\omega(-\Delta)^{\beta}+(-\Delta)^{\beta+1})\ff=(-\Delta)^{\beta}\ff^{3},
		\]
		and the result comes from the decay estimates of the elliptic-type equations \cite{fswx} combined with the heat kernels of $(-\Delta)^m$, $m\in\N$ (see \cite{clyz}).

	\end{proof}

	\begin{remark}
		If $n=1$, then by using the Residue theorem, we can observe for any $\beta>0$ that  
		\begin{equation}\label{residue}
			L_\beta   K_\beta(x)=
			\int_0^\infty
			\frac{\sin(\beta\pi)y^{2\beta}\ee^{-|x|y}}
			{y^{4\beta}(\omega-y^2)+\e^2+2\e y^{2\beta}\cos(\beta\pi)}\dd y.
		\end{equation} 
		So that $\lim_{|x|\to\infty}|x|^{2\beta+1} L_\beta   K_\beta\cong\frac{\sin(\beta\pi)\Gamma(2\beta+1)}{\e^2}$.
		
		If $\beta<1/2$ and $n=1$, then we can obtain for any $a>0$ from   Balakrishnan's representation of $(-\Delta)^\beta$ and Euler's reflection formula that
		\[
		\begin{split}
			(-\Delta)^\beta\ee^{-a|x|}&=\frac{\sin(\beta\pi)}{\pi}\int_0^{+\infty}\frac{-s^{\beta-1}\Delta\ee^{-a|x|} }{s-\da}\dd s\\
			& =
			\frac{\sin(\beta\pi)}{\pi}\int_0^{+\infty} -s^{\beta-1}\Delta\left(\frac{\ee^{-s|x|}}{\sqrt{s}}\ast\ee^{-a|x|}\right)  \dd s\\
			&=\frac{\sin(\beta\pi)}{\pi}\int_0^{+\infty}  \frac{as^{\beta-1}}{s+a^2}\Delta\left( \frac{\ee^{-a|x|}}{a}-\frac{\ee^{-\sqrt{s}|x|}}{\sqrt{s}}\right)  \dd s\\
			&=\frac{\sin(\beta\pi)}{\pi}a^{2\beta}\Gamma(\beta)\Gamma(1-\beta)\ee^{-a|x|}\\
			&\quad-\frac{\sin(\beta\pi)\Gamma(2\beta)}{\pi}a^{2\beta}
			\left[\ee^{\ii a|x|+\frac{(2\beta-1)\ii\pi}{2}}\Gamma(1-2\beta,\ii a|x|)+\ee^{-\ii a|x|-\frac{(2\beta-1)\ii\pi}{2}}\Gamma(1-2\beta,-\ii a|x|)\right]\\  
			&= a^{2\beta} \ee^{-a|x|}\\
			&\quad-\frac{\sin(\beta\pi)\Gamma(2\beta)}{\pi}a^{2\beta}
			\left[\ee^{\ii a|x|+\frac{(2\beta-1)\ii\pi}{2}}\Gamma(1-2\beta,\ii a|x|)+\ee^{-\ii a|x|-\frac{(2\beta-1)\ii\pi}{2}}\Gamma(1-2\beta,-\ii a|x|)\right]
			\\
			& \leq  a^{2\beta} \ee^{-a|x|}, 
		\end{split}
		\]
		where $\Gamma(\cdot,\cdot)$ is the incomplete gamma function. One can  similarly obtain
		that    $L_\beta \ee^{-a|x|}\leq a^{-2\beta} \ee^{-a|x|}$.
		Hence, by combining the above estimate and \eqref{residue}, we deduce  
		$  K_\beta(x)=O(|x|^{-4\beta-1})$ at infinity.

	\end{remark}
	
	\subsection{Uniform boundedness }
	
	Having the existence of the best constant of \eqref{embed-1} in our hands, we now can derive the uniform boundedness of solutions in the energy space $X_\beta$.
	\begin{lemma}[\cite{begout}]
		\label{begout}
		Let $J:=[0,T)\subset\rr$ be a non-degenerated interval. Let $q > 1$, $b>0$ and $a$ be real constants. Define  $\vartheta=(bq)^{-1/(q-1)}$ and
		$f(r)=a-r+br^q$ for $r\geq0$. Let $G(t)$ be a continuous nonnegative function
		on $J$.  If $G(0)<\vartheta$, $a< (1-1/q)\vartheta$ and $f\circ G\geq0$, then
		$G(t)<\vartheta$, for any $t\in J$.
	\end{lemma}

	Now, we are in a position to prove the global existence.

	\begin{proof}[Proof of Theorem \ref{global conditions}]
		In the defocusing case, we have from the energy conservation that
		\[
		2\E(u_0)\geq\|u(t)\|_{\dot{X}_\beta}^2
		\]
		for any $t$ in the existence time interval. In the focusing case, we have from   \eqref{gag-l} that
		\begin{equation}\label{energy-estimate}
			\begin{split}
				2 	\E(u(t))&=\|u(t)\|_{\dot{X}_\beta}^2
				-\frac{1}{2}\|u(t)\|_{L^{4}(\rn)}^{4}\\
				&\geq\|u(t)\|_{\dot{X}_\beta}^2-\frac{ \vr_0}{2}
				\|u\|_\lt^{4-n}
				\| u\|_{\dot{X}_\beta}^{n}.
    			\end{split}
		\end{equation} 
		If $n=1$, then we have from the Young inequality and the mass conservation that
		\[
		\|u(t)\|_{\dot{X}_\beta}^2\lesssim\E(u_0)+\F^2(u_0).
		\]
		In the case $n=2$, the sharp constant $\vr_0$ in \eqref{without-r} implies that $u(t)$ is uniformly bounded in $\dot{X}_\beta$ if
		\[
		\F(u_0)\leq\F(\psi).
		\]
		In the case $n=3$, 
		by   Lemma \ref{begout}, we can define
		$G(t)=\|u(t)\|_{\dot{X}_\beta}^2$ and $f(r)=a-r+br^{q}$, where $a=2\E(u_0)$, $q=3/2$, $b=\frac{\vr_0}{2}\F^{\frac{1}{2}}(u_0).$
		It is obtained from Theorem \ref{thmlwp} that $G$ is continuous. Moreover, we have from   
		\eqref{energy-estimate} that $f\circ G\geq0$. Hence, we have the $\dot{X}_\beta$-uniform bound for the solution if we  can show that $G(0)<\vartheta$ and $a< (1-1/q)\vartheta$, where
		$\vartheta=(bq)^{-1/(q-1)}$.
		Now,  it is easy to see   that  $G(0)<\vartheta$
		is equivalent to \eqref{cond1}. In addition,  it is known  from   \eqref{zero-beta-ground} that (see \cite{cazenave2003}) 
  \begin{equation}\label{poho}
  \|\nabla\psi\|_{L^2(\rr^3)}^2=\frac34\|\psi\|_{L^4(\rr^3)}^4=3\F(\psi),
  \end{equation}
  and
		$$
		\E_0(\psi)=\frac{1}{2}\F(\psi).
		$$
				Therefore, $a< (1-1/q)\vartheta$ is equivalent to \eqref{cond2}. Thus,  
  we get from Lemma
		\ref{begout} and \eqref{poho} that $G(t)<\vartheta$, and equivalently
		 \begin{equation}\label{cond3}
			\|u(t)\|_{\dot{X}_\beta}^2\F (u(t))
			<
			\|\nabla\psi\|_{L^2(\rr^3)}^2\F (\psi).
		\end{equation} 
		Hence, it is concluded from $\F(u(t))=\F(u_0)$  for all $t\in[0,T)$ 
		that $u(t)$ is uniformly bounded in $X_\beta$ for all $[0,T)$.
		
		In the case of the energy critical case $n=4$, we know from the embedding 
		\begin{equation}\label{gn-crit}	 
			\|u\|_{L^4}^{4}\leq C_W\|\nabla u\|_{L^2(\rr^4)}^4
		\end{equation}
		that
		\begin{equation}\label{gn-critical-cn}
			C_{W}=\|W\|_{L^4(\rr^4)}^4\|\nabla W\|_\lt^{-4}=\|\nabla W\|_{L^2(\rr^4)}^{-2^\ast}
		\end{equation} 
		and
		\[
		W(x)= \left(1+\frac{|x|^2}{8}\right) ^{-1}.
		\]
		Hence, we obtain a similar estimate by Lemma \ref{begout} and \eqref{gn-crit}.
		
		In case (v), we have     from \eqref{lemma-best-masszero} that
		\begin{equation} 
			\begin{split}
				2 	\E(u(t))&=\|u(t)\|_{\dot{X}_\beta}^2-\frac{1}{2}\|u(t)\|_{L^{4}(\rn)}^{4}\\
				&\geq\|u(t)\|_{\dot{X}_\beta}^2-\frac{ \vr_\ast}{2}
				\| u(t)\|_{\dot{X}_\beta}^{4}.
			\end{split}
		\end{equation} 
		
		Hence, we obtain  from Lemma \ref{begout} that with $f(r)=2\E(u_0)-r+\frac{ \vr_\ast}{2}r^{2}$ and $r=\|u(t)\|_{\dot{X}_\beta}^2$ that the solution $u(t)$ is bounded in $\dot{X}_\beta$ provided 
		\begin{equation} 
			\|u_0\|_{\dot{X}_\beta} <\vr_\ast^{-1}=C_{\beta,n} \|Q_\ast\|_{\dot{X}_\beta} 
		\end{equation}
		and
		\begin{equation} 
			\E(u_0)<\frac{1}{4}\vr_\ast^{-1}
			=\frac{1}{4}C_{\beta,n,2}^\frac2p\|Q_\ast\|_{\dot{X}_\beta}^2=C_{\beta,n,2 } \E(Q_\ast) 
		\end{equation}
		hold, where in the last inequality we used the fact $4\E(Q_\ast)=\|Q_\ast\|_{\dot{X}_\beta}^2$. This completes the proof.
	\end{proof}
	
	In the focusing case in (v) of Theorem \ref{global conditions}, we can obtain interpolated global conditions involving the mass of the initial data. In this case, equation \eqref{standing} does not possess any scaling, so we are not able to find the sharp constant related to \eqref{embed-1} in terms of the ground states of a suitably scaled equation of \eqref{standing}.
	\begin{theorem}[Interpolated global conditions]\label{Interpolated global conditions}
		Let $u_0\in\x$. Assume in the focusing case that
		\begin{equation}\label{inter-cnd-1}
			\|u_0\|_{\dot{X}_\beta}^{2-m}
			\F^{\frac{\theta+m}{2}}(u_0)<
			\left(\frac{4c_{\beta,m} }{\vr k}\right)^{ \frac{2}{k-2}}
		\end{equation}
		and 
		
		\begin{equation}\label{inter-cnd-2}
			\E^{ k-2 }(u_0)\F^{\frac{(k-2)\theta}{2k}}(u_0)
			<
			c_{\beta,m}^{ k }\left(\frac{k-2}{2k}\right)^{\frac{2(k-2)}{k}}
			\left( \frac{4}{\vr k} \right)^2
		\end{equation}
		where
		\begin{equation}\label{cont-1}
			c_{\beta,m}=(\beta+1)\beta^{-\frac{m\beta}{2(\beta+1)}},
		\end{equation}	
		\begin{equation}\label{cont-2}
			\theta=\frac{4(2-m-\kappa)-n }{ m-2+4\kappa  +n },
		\end{equation}
		\begin{equation}\label{cont-3}
			0\leq \frac m2\leq 1-\kappa-\frac{n }{4},
		\end{equation}
		$k=n+4\kappa $, and $\kappa$ and $\vr$  are the same as in Lemma \ref{embeding-lemma}.
		Then the solution of \eqref{rnls} with initial data $u_0$ is uniformly bounded in $X_\beta$.
	\end{theorem}
	\begin{proof}
		By using the definition of $\E$ in our mind, we can multiply $\E(u)$ by $\F^\frac{\theta}{2}(u_0)$, and then use \eqref{embed-1} and the fact
		\begin{equation}
			\|u\|_{\dot{X}_\beta}^2\geq(\beta+1)\beta^{-\frac{\beta}{\beta+1}}\F(u),
		\end{equation}
		to derive
		\begin{equation}\label{enrg-int-equ}
			\begin{split}
				2\E(u_0)\F^\frac{\theta}{2}(u_0)&=	2 \E(u(t))\F^\frac{\theta}{2}(u(t))\\
				&\geq  c_{\beta,m} 
				\|u(t)\|_{\dot{X}_\beta}^{2-m}
				\F^{\frac{m+\theta}{2}}(u(t))-
				\frac{ \vr}{2}	
				\F^\frac{\theta+4(1-\kappa)-n}{2}(u(t))\|u(t)\|_{\dot{X}_\beta}^{4\kappa +n}.
			\end{split}
		\end{equation}
		Hence, by applying Lemma \ref{begout} with
		$G(t)=\|u(t)\|_{\dot{X}_\beta}^{2-m}\F^{\frac{\theta+m}{2}}(u_0)$, $q=\frac{k}{2}$, $s=2c_{\beta,m}^{-1}\E(u_0)\F^\frac{\theta}{2}(u_0)$ and
		\[
		f(r)=s-r+\frac{ \vr}{2c_{\beta,m}}r^k,
		\]
		we see that
		\[
		\vartheta= \frac{k\vr}{4c_{\beta,m} } ,
		\] 
		and thereby  after a straightforward computation it is reveals that $G(0)<\vartheta$ and $s<\frac{k-2}{k}\vartheta$ are equivalent to \eqref{inter-cnd-2} and \eqref{inter-cnd-2}. 
	\end{proof}
	
	\begin{remark}\label{remark-cons}
		Note for    $m=2$ that $c_{\beta,2}\leq2$  and formally $\lim_{\beta\to+\infty}c_{\beta,2}=\lim_{\beta\to0}c_{\beta,2}=1$.
	\end{remark}

	\section*{Appendix} 
	\setcounter{equation}{0}
	\renewcommand{\theequation}{A.\arabic{equation}}
	\setcounter{theorem}{0}
	
	\renewcommand{\thetheorem}{A.\arabic{theorem}}

	In this appendix, we give the proof for the dispersive estimate in Lemma \ref{lm-dispest}. To do this we need Van der Corput Lemma (see e.g., \cite{stein}).

	\begin{lemma}[Van der Corput Lemma]\label{lm-corput}
		Assume 
		$g \in C^1(a, b)$, $\psi\in  C^2(a, b)$ and $|\psi''(r)|  \ge  A$ for all $r\in (a, b)$. Then 
		\begin{align}
			\label{corput} 
			\Bigabs{\int_a^b \ee^{\ii t  \psi(r)}   g(r) \d r}& \le C  (At)^{-1/2}  \left[ |g(b)| + \int_a^b |g'(r)| \d r \right] ,
		\end{align}
		for some constant $C>0$ that is independent of $a$, $b$ and $t$.
	\end{lemma}

	Lemma \ref{lm-corput} holds even if $\psi'(r)=0$ for some $r\in (a, b)$.
	However, if $|\psi'(r)|> 0$ for all $r\in (a, b)$, one can obtain the following lemma by using integration by parts. 
	\begin{lemma}[\cite{ddt}]\label{lm-corput1}
		Suppose that $g \in C^\infty_0(a,b)$ and $\psi\in C^\infty (a, b)$ with $|\psi'(r)|> 0$ for all $r\in (a, b)$. If 
		\begin{equation}\label{dervbd}
			\max_{a\le r\le b}\abso{\frac{\d^j}{\d r^j}g(r)} \leq A, \qquad 
			\max_{a\le r\le b}  \abso{\frac{\d^j}{\d r^j}\left( \frac 1{\psi'(r)} \right)} \leq B
		\end{equation}
		for all $j=0,\cdots,N\in\N$,
		then 
		\begin{align}
			\label{corput'} 
			\abso{\int_a^b \ee^{\ii t  \psi(r)}   g(r) \d r}& \lesssim A B^N  |t|^{-N} .
		\end{align}
	\end{lemma}

	\subsection*{Proof of Lemma \ref{lm-dispest}}

	Without loss of generality, we may assume $t>0$. Now,
	we can write
	\[
	\left[  \ee^{\ii tm(D)} f_\lambda \right](x) 
	= (I(\cdot, t)\ast f)(x),
	\]
	where
	\begin{equation}\label{Idef}
		I_{\lambda} (x, t) 
		= \lambda^n \int_{\R^n} \ee^{\ii (\lambda  x \cdot  \xi+   t  m( \lambda |\xi|)}  \rho(|\xi|)) \dd\xi,
	\end{equation}
	By Young's inequality, it suffices to prove
	\begin{equation}\label{Iest}
		| I_{\lambda} (x, t) |\le C t^{-\frac n2}.
	\end{equation}

	We provide the proof only for $n\ge 2$ as the one-dimensional case ($n=1$) is easier to treat.
	Using polar coordinates, we can write
	\begin{equation}\label{I-eq}
		I_{\lambda} (x, t) =\lambda^n  \int_{1/2}^2  \ee^{\ii t m(\lambda r) }  (\lambda r|x|)^{-\frac{n-2}{2}}  J_{\frac{n-2}{2}}( \lambda r |x|)   r^{n-1}  \rho(r) \dd r,
	\end{equation}
	where $J_k(r)$ is the Bessel function:
	$$
	J_k(r)=\frac{ (r/2)^k}{(k+1/2) \sqrt{\pi}} \int_{-1}^1  \ee^{\ii r s} \left(1-s^2\right)^{k-1/2} \d s \quad \text{for} \ k>-1/2.
	$$
	
	The Bessel function $J_k(r)$ satisfies the following properties for $k>-1/2$ and $r>0$,
	\begin{align}
		\label{Jm1}
		J_k (r) &\le Cr^{k} ,
		\\
		\label{Jm2}
		J_k(r)& \le C r^{-1/2} ,
		\\
		\label{Jm3}
		\frac{\dd  }{\dd r} \left[ r^{-k} J_k(r)\right] &= -r^{-k} J_{k+1}(r)
	\end{align}
	
	Moreover, we can write
	\begin{equation}\label{J0est}
		r^{- \frac{n-2}2 }J_{ \frac{n-2}2}(s)= \ee^{\ii s} h(s)  +\ee^{-\ii s}\bar h(s)
	\end{equation}
	for some function $h$ satisfying the decay estimate 
	\begin{equation}\label{h-est}
		\abso{\frac{\d^j}{\d r^j} h(r)} \le C_j \angles{r}^{-\frac{n-1}2-j}  \quad \text{for all} \ j\ge 0. 
	\end{equation}

	We use the short-hand
	$$ 
	m_{\lambda}(r) = m(\lambda r),  \qquad \tilde J_a(r)= r^{-a} J_a(r), \qquad \tilde\rho(r)=r^{n-1} \rho(r).$$
	
	Hence,
	\begin{equation}\label{I-eqq}
		I_{\lambda} (x, t) = \lambda^n  \int_{1/2}^2  \ee^{\ii t  m_{\lambda}(r)}  \tilde J_{\frac{n-2}{2}}( \lambda r |x|)  \tilde  \rho(r) \dd r.
	\end{equation}

	\vspace{2mm}

	\subsubsection{\underline{Case 1: $  |x|\lesssim  \lambda^{-1}$}:}
	
	By \eqref{J0est} and \eqref{h-est} we have for all $ r\in (1/2, 2)$ the estimate
	\begin{equation}
		\label{J0derv-est}
		\left| \partial_r ^j \left[  \tilde J_{ \frac{n-2}2 }( \lambda r |x|)  \tilde\rho(r) \right]\right| \underset{j}  \lesssim 1  \qquad  ( j \ge 0).
	\end{equation}
	
	We have for $\lambda \gg 1$,
	\begin{equation}\label{mlamb-invest}
		\max_{ 1/2 \le r \le 1 }\abso{\frac{\d^j}{\d r^j} \left(  \frac 1{  m'_{ \lambda}(r)  } \right)} \underset{j}  \lesssim \lambda^{- 2}  \qquad   (j \ge 0).
	\end{equation}

	Applying Lemma \ref{lm-corput1} with \eqref{J0derv-est}-\eqref{mlamb-invest} and $N=  n/2$
	to \eqref{I-eqq}, we obtain
	\begin{equation}\label{Iest-2}
		\begin{split}
			| I_{\lambda} (x, t) |
			&\lesssim \lambda^n \cdot \lambda^{- n} t^{-   \frac n2}  \lesssim  t^{-   \frac n2} .
		\end{split}
	\end{equation}

	\subsubsection{\underline{Case 2: $  |x|\gg  \lambda^{-1}$.}}
	Using \eqref{J0est} in \eqref{I-eq} we write
	\begin{align*}
		I_{\lambda} (x, t) 
		&=\lambda^n \left\{\int_{1/2}^2  \ee^{\ii t \phi^+_{\lambda} (r)  }  h(\lambda r |x|)  \tilde \rho(r) \d r +  \int_{1/2}^2  \ee^{-\ii t \phi^-_{\lambda} (r)  } \bar  h(\lambda r |x|)  \tilde\rho(r) \d r \right\},
	\end{align*}
	where 
	$$
	\phi^\pm_{\lambda} (r)=    \lambda r|x|/t  \pm m(\lambda  r).
	$$
	Set $H_{\lambda}( |x|, r) =h(\lambda r |x|)  \tilde \rho(r)$. In view of \eqref{h-est} we have 
	\begin{equation}
		\label{Hest}
		\max_{1/2 \le r\le 2}\Bigabs {\partial_r ^j H_{\lambda}( |x|, r)  }    \lesssim   (\lambda |x|)^{-\frac{n-1}2}  \qquad  ( j \ge 0),
	\end{equation}
	where we also used the fact $\lambda |x|\gg 1$

	Now 
	we write 
	$$
	I_{\lambda} (x, t) 
	= I^+_{\lambda} (x, t) 
	+  I^-_{\lambda} (x, t) ,
	$$
	where
	\begin{align*}
		I^+_{\lambda} (x, t) 
		&= \lambda^n \int_{1/2}^2  \ee^{\ii t \phi^+_{\lambda} (r)  } H_{\lambda}( |x|, r)    \d r ,
		\\
		I^-_{\lambda} (x, t) &= \lambda^n
		\int_{1/2}^2  \ee^{-\ii t \phi^-_{\lambda} (r)  }\bar H_{\lambda}( |x|, r)    \d r .
	\end{align*}
	
	Observe that
	$$
	\partial_r \phi^\pm_{\lambda} (r)=  \lambda  \left[ |x|/t \pm  m'(\lambda r) \right],\qquad \partial_r^2\phi^\pm_{\lambda} (r)=     \pm  \lambda^2 m''(\lambda r),
	$$
	and hence
	\begin{equation}
		\label{phi'+:est}
		|\partial_r \phi^+_{\lambda} (r)|\gtrsim  \lambda^2,
		\qquad 
		|\partial^2_r \phi^\pm_{\lambda} (r)| \sim   \lambda^2
	\end{equation}
	for all $ r\in (1/2, 2)$, where we also used the fact that $\lambda \gg 1$.

	\subsubsection*{\underline{Estimate for  $I^+_{\lambda} (x, t)$ } }
	
	Analogous to estimate \eqref{mlamb-invest}, we notice that
	\begin{equation}
		\label{Est-phi+'-1}
		\max_{1/2 \le r \le 2}\Bigabs {\partial_r ^j \left( \left[ \partial_r \phi_\lambda^+ (r) \right]^{-1} \right)}  \underset{j} \lesssim \lambda^{- 2}\qquad   (j \ge 0).
	\end{equation} 
	Applying Lemma \ref{lm-corput1} with \eqref{Hest}, \eqref{Est-phi+'-1}  and $N=  n/2$
	to $I^+_{\lambda} (x, t) $ we obtain
	\begin{equation}\label{Iest-3}
		\begin{split}
			| I^+_{\lambda} (x, t) |
			&
			\lesssim  \lambda^n  \cdot  (\lambda |x|)^{-\frac{n-1}2}  \cdot \lambda^{- n}  t^{- \frac n2 }
			\lesssim   t^{- \frac n2} ,
		\end{split}
	\end{equation}
	where we also used the fact that  $\lambda |x|\gg 1$.

	\subsubsection*{\underline{Estimate for $I^-_{\lambda} (x, t)$}}
	We treat the non-stationary and stationary cases separately. In the non-stationary
	case, where 
	$$ |x | \ll   \lambda t \quad \text{or} \quad |x| \gg   \lambda   t, $$ we have 
	$$
	|\partial_r \phi^-_{\lambda} (r)|\gtrsim   \lambda ^{2},
	$$
	and hence $I^-_{\lambda} (x, t)$ can be estimated in exactly the same way as $I^+_{\lambda} (x, t)$ above, and satisfies the same bound.

	So it remains to treat the stationary case: 
	$$
	|x | \sim  \lambda  t.
	$$
	In this case,  
	we use Lemma \ref{lm-corput}, \eqref{phi'+:est} and   \eqref{Hest} to obtain 
	\begin{equation}\label{Iest-station}
		\begin{split}
			| I^-_{\lambda} (x, t) 
			&\lesssim \lambda^n \left( \lambda^2  t \right)^{-\frac12}\left[ \abso{H_\lambda^- (x, 2)}+ \int_{1/2}^2 \abso{ \partial_r  H_\lambda^- (x, r)} \dd r\right]
			\\
			&\lesssim   \lambda^{n-1 }  t^{-\frac12} \cdot (\lambda |x|)^{-\frac {n-1}2}
			\\
			& \lesssim     t^{-\frac n2},
		\end{split}
	\end{equation}
	where we also used the fact that $H_\lambda^- (x, 2) =0$.

	\subsection*{Conflict of interest} The authors declare that they have no conflict of interest. 
	
	\subsection*{Data Availability}
	There is no data in this paper.
 
	\section*{Acknowledgment}
	
	A. E. is supported by Nazarbayev University under the Faculty Development Competitive Research Grants Program for 2023-2025 (grant number 20122022FD4121).
	


\end{document}